\documentclass[a4paper,11pt]{article}

\usepackage{amsmath}
\usepackage{amsfonts}
\usepackage{amssymb}
\usepackage{amstext    }
\usepackage{amsthm    }

\usepackage[a4paper]{geometry}
\usepackage{mathrsfs}

\usepackage{hyperref}

\usepackage{color}

\usepackage[T1]{fontenc}
\usepackage{graphicx}
\usepackage[all,cmtip]{xy}

\usepackage{epsf}
\usepackage{psfrag}

\usepackage{fancyhdr}

\usepackage{url}

\newtheorem{theorem}{Theorem}[section]
\newtheorem{lemma}[theorem]{Lemma}
\newtheorem{definition}[theorem]{Definition}
\newtheorem{proposition}[theorem]{Proposition}
\newtheorem{remark}[theorem]{Remark}
\newtheorem{corollary}[theorem]{Corollary}

\newtheorem{example}[theorem]{Example}
\newtheorem{question}{Question}
\newcommand{\cali}[1]{\mathscr{#1}}

\numberwithin{equation}{section}

\newcommand{\dist}{{\rm dist}}
\newcommand{\ddc}{{dd^c}}

\newcommand{\Jac}{{\rm Jac}}
\newcommand{\supp}{{\rm supp}}
\newcommand{\id}{{\rm Id}}
\newcommand{\aut}{{\rm Aut}}

\newcommand{\Ac}{\cali{A}}

\newcommand{\Cc}{\cali{C}}
\newcommand{\Dc}{\cali{D}}

\newcommand{\Fc}{\cali{F}}

\newcommand{\Hc}{\cali{H}}
\newcommand{\Ic}{\cali{I}}
\newcommand{\Jc}{\cali{J}}

\newcommand{\Nc}{\cali{N}}

\newcommand{\Pc}{\cali{P}}

\newcommand{\Vc}{\cali{V}}

\newcommand{\Bb}{\mathbb{B}}
\newcommand{\Db}{\mathbb{D}}
\newcommand{\Pb}{\mathbb{P}}

\newcommand{\Cb}{\mathbb{C}}

\newcommand{\Nb}{\mathbb{N}}

\newcommand{\Rb}{\mathbb{R}}
\newcommand{\card}{\text{card}}

\usepackage[T1]{fontenc}
\usepackage[english]{babel}

\author{Johan Taflin}

\title{Attracting Currents and Equilibrium Measures for Quasi-attractors of $\Pb^k$}
\begin{document}
\maketitle

\begin{abstract}
Let $f$ be a holomorphic endomorphism of $\Pb^k$ of degree $d.$  For each quasi-attractor of $f$ we construct a finite set of currents with attractive behaviors. To every such an attracting current is associated an equilibrium measure which allows for a systematic ergodic theoretical approach in the study of quasi-attractors of $\Pb^k.$  As a consequence, we deduce that there exist at most countably many quasi-attractors, each one with topological entropy equal to a multiple of $\log d.$ We also show that the study of these analytic objects can initiate a bifurcation theory for attracting sets.
\end{abstract}
\section{Introduction}
If $f$ is a continuous self-map of a topological space $X$ there are several notions of attractors for the dynamical system defined by the iterates $f^n:=f\circ\cdots\circ f$ of $f.$ In this paper, we will only consider the following topological ones. A compact subset $A$ of $X$ is called an \textit{attracting set} if it has a non-empty open neighborhood $U,$ called a \textit{trapping region}, such that $f(U)\Subset U$ and $A=\cap_{n=0}^\infty f^n(U).$ An \textit{attractor} is an attracting set which possesses a dense orbit. Hurley \cite{hurley} also introduced the notion of \textit{quasi-attractors} which are intersections of attracting sets. They are related to Conley's \textit{chain recurrence classes} \cite{conley}. The minimal chain recurrence classes are exactly the minimal quasi-attractors which therefore correspond to the chain recurrent quasi-attractors. Unlike attractors, there always exists at least one chain recurrent quasi-attractor when $X$ is compact.

In this paper, we consider the case where $X$ is the complex projective space $\Pb^k$ and $f$ is a holomorphic endomorphism of algebraic degree $d\geq2.$ This situation has been studied by many authors, see e.g. \cite{fs-example, fw-attractor, jonsson-weickert-attractor, d-attractor, rong, daurat}. In \cite{d-attractor}, Dinh constructed, under geometric assumptions, an \textit{attracting current} and an \textit{equilibrium measure} associated to an attracting set. Our first aim is to generalize this construction to all quasi-attractors without any assumption. These analytic objects will give us interesting information about the dynamics on a quasi-attractor and also about its structure. Another motivation was to study attracting sets for holomorphic families of endomorphisms.
In one complex variable, the bifurcation locus of a holomorphic family of rational maps on $\Pb^1$ is the closure of the set of parameters where an attracting cycle appears or disappears. When this happens, there still exists at the limit a periodic cycle, which is neural. In higher dimension, the situation is more complicated. It seems more difficult to follow an attracting set and to study directly the limit set when it ``disappears''. It could therefore be easier to consider the analytic objects associated to them and their numerical invariants. For example, in Section \ref{sec-ac-bif} we give criteria to determine if an attracting set moves continuously with the parameter.

To introduce our results, we have to fix some notations and recall classical statements. Denote by $\omega$ the standard Fubini-Study form on $\Pb^k.$ If $E$ is a subset of $\Pb^k$ then we define $\Cc_p(E)$ to be the set of positive closed currents $S$ of bidegree $(p,p),$ supported in $E$ and of mass $1,$ i.e. $\langle S,\omega^{k-p}\rangle=1.$ Observe that $\Cc_k(E)$ corresponds to the set of probability measures supported in $E.$

From now on, $f$ will be an endomorphism of $\Pb^k$ of algebraic degree $d\geq2.$ The \textit{Green current of order $p$} of $f$ is defined by $T^p:=\lim_{n\to\infty}d^{-np}(f^n)^*\omega^p.$ Its support $\Jc_p$ is called the \textit{Julia set of order $p$}. These Julia sets give a filtration
$$\Pb^k=:\Jc_0\supset \Jc_1\supset \Jc_2\supset\cdots\supset\Jc_k\supset\Jc_{k+1}:=\varnothing,$$
which will be of particular importance in the sequel. It is a result of Forn{\ae}ss-Sibony \cite{fs-cdhd2} and Ueda \cite{ueda-fatou} that $\Jc_1$ is equal to the Julia set $\Jc$ of $f$ i.e. the complement of the largest open set where the family of iterates $(f^n)_{n\geq1}$ is normal. The set $\Jc_k$ is the support of the \textit{equilibrium measure} of $f,$ $\mu:=T^k,$ which has many interesting dynamical properties. However, a quasi-attractor which intersects $\Jc_k$ has to be equal to the whole space $\Pb^k.$ Therefore, this measure $\mu$ gives no information about proper quasi-attractors.

As we will see later, our results and techniques have similarities with those used in the study of horizontal-like maps \cite{ds-geo} (see also \cite{dujardin-henon}, \cite{dns-hori}). But our geometric setting is a priori very different and we need a better understanding of it. In \cite{daurat}, Daurat introduced a notion of \textit{dimension} for an attracting set which stays pertinent for quasi-attractors. A quasi-attractor $A$ has dimension $s$ if $\Cc_{k-s}(A)\neq\varnothing$ and $\Cc_{k-s-1}(A)=\varnothing.$ Our first result is to give an equivalent definition.
\begin{proposition}\label{prop-intro}
A quasi-attractor $A$ has dimension $s$ if and only if $A\cap\Jc_s\neq\varnothing$ and $A\cap\Jc_{s+1}=\varnothing.$
\end{proposition}
Although simple, this fact is essential since it provides us with geometric information about $A.$ If $A$ is a quasi-attractor of dimension $s$ then $T^{s+1}$ belongs to $\Cc_{s+1}(\Pb^k\setminus A)$ which implies that $A$ is weakly $(k-s)$-pseudoconvex (see Section \ref{sec-pluri}). This will allow us to use the $\ddc$-method developed by Dinh and Sibony to study the dynamics induced by $f$ on $\Cc_{k-s}(U)$ for a trapping region $U$ of $A.$ In particular, we obtain that there exists on each attracting set at least one current which exhibits equidistribution properties.
\begin{theorem}\label{th equi as}
Let $A$ be an attracting set of dimension $s$ with a trapping region $U.$ There exist a trapping region $D_\tau\subset U$ and a current $\tau$ in $\Cc_{k-s}(D_\tau)$ such that
$$\lim_{N\to\infty}\frac{1}{N}\sum_{n=0}^{N-1}\frac{1}{d^{ns}}(f^n)_*R=\tau$$
for all continuous currents $R$ in $\Cc_{k-s}(D_\tau).$
\end{theorem}
We say that $\tau$ is an \textit{attracting current} of bidimension $(s,s)$ associated to $A$ (see Definition \ref{def-attracting-current} for a precise definition). They can be considered as attractive periodic points in the set $\Cc_{k-s}(\Pb^k).$ In general, there exist several such currents supported in $A$ and it is not possible to remove the Ces\`{a}ro mean in the theorem. For example, if $A=\{p_0,p_1,p_2\}$ is the union of an attracting fixed point $p_0$ and an attracting cycle $\{p_1,p_2\}$ of period $2$ then $A$ is an attracting set of dimension $0$ with two attracting measures, the Dirac mass $\delta_{p_0}$ and $2^{-1}(\delta_{p_1}+\delta_{p_2}).$ For the latter, if $R$ is a smooth probability supported in a small neighborhood of $p_1$ then the sequence $(f^n)_*R$ has two different limit values, $\delta_{p_1}$ and $\delta_{p_2}.$ However, if we exchange $f$ by $f^2$ in this example then the sequence $(f^n)_*R$ converges and we obtain three attracting measures which cannot be decomposed anymore. This phenomenon still holds in the general case.

\begin{theorem}\label{th finitude}
Let $A$ be an attracting set of dimension $s.$ The set of attracting currents of bidimension $(s,s)$ supported in $A$ is finite. Moreover, there exists an integer $n_0\geq1$ such that if $\tau$ is an attracting current of bidimension $(s,s)$ supported in $A$ for $f^{n_0}$ then 
$$\lim_{n\to\infty}\frac{1}{d^{nn_0s}}(f^{nn_0})_*R=\tau$$
for all continuous currents $R$ in $\Cc_{k-s}(D_\tau).$ Here, $D_\tau$ is the trapping region associated to $\tau$ and $f^{n_0}$ by Theorem \ref{th equi as}.
\end{theorem}
A key ingredient in the proofs of these two theorems is to consider positive closed currents as geometric objects and to use them in order to build new trapping regions.

In the example above, at the end each connected component of $A$ is the support of an attracting measure. In general, it is not enough to consider the connected components but it is possible to associate to each current $\tau$ in Theorem \ref{th finitude} an attracting set $A_\tau$ which can be seen as the ``irreducible component'' of $A$ containing $\tau,$ cf. Definition \ref{def-irreducible-compo}. When $A$ is algebraic, this definition coincides with the classical one and $A$ is the union of its irreducible components. It is likely that such a decomposition holds in general.

We now come back to quasi-attractors. If $A$ is a quasi-attractor of dimension $s$ then, by definition, $A$ is equal to the intersection of attracting sets $\{A_i\}_{i\geq0}.$ As the intersection of two attracting sets is again an attracting set, we can assume that the sequence $\{A_i\}_{i\geq0}$ is decreasing and that each of them is of dimension $s.$ A fundamental question is to know if all these sequences have to be stationary i.e. if all quasi-attractors are indeed attracting sets. In the real setting, it is easy to create an example with a quasi-attractor which is not an attracting set. In real dimension larger or equal to three, Bonatti-Li-Yang \cite{bonatti-li-yang} prove that it is even possible to find a residual subset $\mathcal U$ of an open set of $\mathcal C^r$ diffeomorphisms, $r\geq1,$ such that if $f\in\mathcal U$ then none of the minimal quasi-attractors of $f$ is an attracting set. It means that this phenomenon can have some kind of robustness. In the holomorphic setting, we are only able to exclude it in some special cases. For example, if $A=\cap_{i\geq0}A_i$ is such a quasi-attractor in $\Pb^2$ then the Hausdorff dimension of each $A_i$ has to be larger or equal to $3$ (see Remark \ref{rk-qa-as}). However, a consequence of the finiteness in Theorem \ref{th finitude} is that, from the point of view of currents, attracting sets and quasi-attractors are the same.
\begin{corollary}\label{cor-qa-ca}
If $A$ is equal to the intersection of a decreasing sequence of attracting sets $\{A_i\}_{i\geq0}$ of dimension $s$ then there exists an integer $i_0\geq0$ such that the attracting currents of $A_i$ are equal to those of $A_{i_0}$ for all $i\geq i_0.$ Moreover, the minimal elements in the set of dimension $s$ quasi-attractors are in one-to-one correspondence with the set of attracting currents in $\Cc_{k-s}(\Pb^k).$ In particular, a holomorphic endomorphism of $\Pb^k$ admits at most countably many minimal quasi-attractors.
\end{corollary}
We define the attracting currents of the quasi-attractor $A$ to be those of the attracting set $A_{i_0}.$ Regarding the cardinality of the set of quasi-attractors, by a result of Gavosto \cite{gavosto} (see also \cite{buzzard}) if $k\geq2$ then a holomorphic endomorphism of $\Pb^k$ can have infinitely many sinks, thus in particular infinitely many attractors. So the bound in the corollary is sharp and we do not know a direct proof of it when $k\geq3.$ By cohomological arguments, there exists at most one minimal quasi-attractor of dimension larger or equal to $k/2.$ Hence, the case $k=2$ follows easily from the definition of the dimension (cf. \cite{fw-attractor} for a different proof). Again, this result is not true in the real setting. Bonatti and Moreira remark that the unfolding of a homoclinic tangency can create, in a residual subset of an open set of diffeomorphisms, uncountably many minimal quasi-attractors.

From an invariant current, it is classical to build an invariant measure (see e.g. \cite{fs-example}). Let $\tau$ be an attracting current of bidimension $(s,s).$ We define the \textit{equilibrium measure} $\nu_\tau$ associated to $\tau$ as the intersection of $\tau$ with the Green current of order $s$:
$$\nu_\tau:=\tau\wedge T^s.$$
Using previous results from \cite{bs-3}, \cite{dethelin-lyap}, \cite{d-attractor}, it is easy to deduce from Theorem \ref{th equi as} and Theorem \ref{th finitude} the following properties for $\nu_\tau.$
\begin{corollary}\label{cor-qa-nu}
Let $A$ be a quasi-attractor of dimension $s$ for $f.$ If $\tau$ is an attracting current in $\Cc_{k-s}(A)$ then its equilibrium measure $\nu_\tau$ is an ergodic measure of maximal entropy $s\log d$ on $A$ and it has at least $s$ positive Lyapunov exponents. Moreover, if $n_0$ is the integer defined in Theorem \ref{th finitude} then $\nu_\tau$ has at most $n_0$ ergodic components with respect to $f^{n_0},$ each of which is mixing.
\end{corollary}
In particular, the topological entropy of a quasi-attractor is always a multiple of $\log d$ and this multiple is exactly the dimension of the quasi-attractor. Observe that this dimension is a geometric invariant which is independent of $f.$ It is easy to see that quasi-attractors of dimension $0$ are finite unions of sinks. Therefore, a quasi-attractor has zero topological entropy if and only if it is a union of sinks. Another consequence is that if $f$ is not chain recurrent on $\Pb^k$ and possesses no attractive periodic orbit then such an equilibrium measure must be supported in $\Jc\setminus\Jc_k.$

Unfortunately, we were not able to prove that $\nu_\tau$ is always hyperbolic. It is the case in all known examples cf. \cite{t-attrac-speed}, \cite{da-ta}. If we knew that the convergence in Theorem \ref{th finitude} has exponential speed then the hyperbolicity of $\nu_\tau$ would follow from an easy adaptation of arguments in \cite{dethelin-lyap}. The property of $\nu_\tau$ to be hyperbolic would have several consequences. For example, in \cite{daurat-2} Daurat proved in this situation and when $k=2$ that $\nu_\tau$ represents an equidistribution of saddle periodic points in $D_\tau$ and that the Green current $T$ is laminar in $D_\tau.$ Moreover, if $\tau$ were the unique invariant current in $\Cc_{k-s}(D_\tau)$ then techniques developed by Bedford-Lyubich-Smillie \cite{bls-4} would imply that $\nu_\tau$ is the unique measure of entropy $s\log d$ on $D_\tau$ (see \cite{daurat-2}). These two questions, the speed of convergence in Theorem \ref{th finitude} and the uniqueness of $\tau,$ are related and we postpone them to a later work. See Section \ref{subsec-speed} for results in that direction.

In order to better illustrate the above results we consider now the case of minimal quasi-attractors, i.e. quasi-attractors which are chain recurrent. We emphasize that these dynamically meaningful objects have the advantage to always exist and that attractors are special cases of them.
\begin{corollary}\label{cor-qa-final}
Let $K$ be a minimal quasi-attractor of dimension $s.$ There exists an integer $n_0\geq1$ such that if we exchange $f$ by $f^{n_0}$ then $K$ splits into $n_0$ minimal quasi-attractors $K=K_1\cup\cdots\cup K_{n_0}$ such that each $K_i$ is contained in a trapping region $U_{K_i}$ which supports a unique attracting current $\tau_i\in\Cc_{k-s}(K_i)$ with
$$\lim_{n\to\infty}\frac{1}{d^{ns}}(f^n)_*R=\tau_i$$
for all continuous currents $R$ in $\Cc_{k-s}(U_{K_i}).$ Moreover, the equilibrium measure $\nu_{\tau_i}$ is mixing, of maximal entropy $s\log d$ on $K_i$ and has at least $s$ positive Lyapunov exponents.
\end{corollary}
In particular, this result implies that each set $K_i$ as above is still a minimal quasi-attractor for all iterates of $f.$ This prevents phenomena such as adding machines to appear.
Adding machines are simples examples of minimal quasi-attractors without periodic orbit and which split into arbitrarily large number of pieces when we exchange $f$ by an iterates.

Finally, we study attracting sets for holomorphic families of endomorphisms. We show that the number of attracting currents of bidimension $(s,s)$ is constant in a family and that we can glue them together in order to form structural varieties. In particular, we obtain that the sum of the Lyapunov exponents of the equilibrium measure $\nu_\tau$ is a plurisubharmonic function of the parameter, see Corollary \ref{cor-sum-lyap}. This suggests that we can undertake a bifurcation theory of attracting sets in $\Pb^k$ using positive closed currents defined on the parameter space by these plurisubharmonic (p.s.h) functions.

The paper is organized as follows. In Section \ref{sec-pluri}, we introduce the notions of weakly $p$-pseudoconvex sets and of structural varieties which provide us with the tools to study the set $\Cc_{k-s}(U).$ Then, Section \ref{sec courant} is devoted to the dynamics in this space of currents. We establish there the existence and the finiteness of attracting currents as well as several equidistribution results toward them. Theorem \ref{th equi as} is a consequence of Theorem \ref{th conv1} and Theorem \ref{th finitude} comes from a combination of Theorem \ref{th finitude des tau} and Theorem \ref{th convergence dans V}. This allows us in Section \ref{sec-measure} to apply classical methods in order to study the equilibrium measures. In Section \ref{sec-ac-bif}, we define the irreducible components of an attracting set and we investigate their behavior in families. In the last section, we consider more specifically the case of quasi-attractors and we prove the three corollaries state above. We also suggest a list of open questions.

The author would like to thank Fabrizio Bianchi, Christian Bonatti and Tien-Cuong Dinh for helpful discussions or constructive comments.

\section{Basic facts about pluripotential theory and geometry of $\Pb^k$}\label{sec-pluri}
All the results in this section are well-known and we refer to \cite{de-book} for an introduction to pluripotential theory. First, we recall the notion of weakly $p$-pseudoconvex sets. As we will see in the next section, a trapping region in $\Pb^k$ is always weakly $p$-pseudoconvex for some $p$ which will give us a geometric setting for the sequel. In the second part of this section we give simple results about structural varieties in $\Pb^k.$ From now on, and for the rest of the paper, $k,$ $p$ and $s$ are three non-negative integers such that $s=k-p.$

\subsection{Weakly $p$-pseudoconvex domains}\label{subsec p-pseudo}
Dinh and Sibony introduced the following definition (cf. \cite{ds-superpot}). Following their conventions, in the sequel positivity for forms and currents corresponds to strong positivity in \cite{de-book}.
\begin{definition}
A compact subset $K$ of a complex manifold $X$ of dimension $k$ is weakly $p$-pseudoconvex if there exists a positive smooth $(s,s)$-form $\phi$ such that $\ddc\phi$ is strictly positive on $K.$
\end{definition}
It is easy to see that a compact subset $K$ of $\Pb^k$ such that $\Cc_{s+1}(\Pb^k\setminus K)$ is not empty is weakly $p$-pseudoconvex in $\Pb^k$ (see \cite{ds-superpot}). Indeed, Forn\ae ss and Sibony show that such a set is $(p-1)$-pseudoconvex (cf. \cite{fs-oka} for definitions and mind the difference in indices).

An important point about weakly $p$-pseudoconvex compact sets in $\Pb^k$ is that a current of bidegree $(p,p)$ supported on such a set is totally determined by its values on smooth forms $\phi$ with $\ddc\phi\geq0.$ To be more precise, let $U$ be an open set of $\Pb^k$ such that there exists a positive smooth $(s,s)$-form $\phi_0$ with $\ddc\phi_0\geq\omega^{s+1}$ on $\overline U.$ We denote by $\Pc(U)$ the set of $\mathcal C^2$ $(s,s)$-forms $\phi$ on $U$ such that $0\leq\phi\leq\omega^s$ and $\ddc\phi\geq0.$

\begin{lemma}\label{le plongement de C2}
There exists a constant $c>0$ such that if $\phi$ is a $(s,s)$-form of class $\mathcal C^2$ on $\Pb^k$ then the restriction $\widetilde \phi$ to $U$ of
$$\frac{1}{c+2}\left(\frac{\phi}{c\|\phi\|_{\mathcal C^2(\Pb^k)}}+\phi_0+\omega^s\right)$$
is in $\cali P(U).$
\end{lemma}
\begin{proof}
There exists a constant $c>1$ such that $-c\|\phi\|_{\mathcal C^2(\Pb^k)}\omega^s\leq\phi\leq c\|\phi\|_{\mathcal C^2(\Pb^k)}\omega^s$ and $-c\|\phi\|_{\mathcal C^2(\Pb^k)}\omega^{s+1}\leq\ddc\phi\leq c\|\phi\|_{\mathcal C^2(\Pb^k)}\omega^{s+1}$ on $\Pb^k,$ for all $(s,s)$-forms of class $\mathcal C^2.$ Therefore
$$\ddc\left(\frac{\phi}{c\|\phi\|_{\mathcal C^2(\Pb^k)}}+\phi_0\right)\geq0,$$
on $U.$ Moreover, if $c>1$ is large enough, $\phi_0\leq c\omega^s$ and thus $\widetilde\phi$ defined above satisfies both conditions in the definition of $\Pc(U).$
\end{proof}
As a direct consequence, we have the following lemma.
\begin{lemma}\label{le sepa}
If $R$ and $S$ are two $(p,p)$-currents supported on $U$ such that
$$\langle R,\psi\rangle=\langle S,\psi\rangle$$
for all $\psi\in\Pc(U)$ then $R=S.$
\end{lemma}
\begin{proof}
If $\phi$ is a $\mathcal C^2$ $(s,s)$-form on $\Pb^k$ then by Lemma \ref{le plongement de C2}
$$\left\langle R,\frac{\phi}{c\|\phi\|_{\mathcal C^2(\Pb^k)}}+\phi_0+\omega^s\right\rangle=\left\langle S,\frac{\phi}{c\|\phi\|_{\mathcal C^2(\Pb^k)}}+\phi_0+\omega^s\right\rangle.$$
Moreover, $\omega^s$ and $c^{-1}\phi_0$ are also in $\Pc(U),$ hence $\langle R,\phi_0\rangle=\langle S,\phi_0\rangle$ and $\langle R,\omega^s\rangle=\langle S,\omega^s\rangle.$ Therefore, $\langle R,\phi\rangle=\langle S,\phi\rangle.$
\end{proof}
In the dynamical part of this paper, we are interested to open subsets $U$ of $\Pb^k$ such that $\Cc_{s+1}(\Pb^k\setminus U)$ and $\Cc_p(U)$ are both non-empty. But to study $\Cc_p(U)$ it is enough to consider the smallest open subset $\widetilde U\subset U$ such that $\Cc_p(\widetilde U)=\Cc_p(U).$
\begin{definition}\label{def pseudoconcave core}
Let $U$ be an open subset of $\Pb^k$ such that $\Cc_p(U)\neq\varnothing.$ We define the $s$-pseudoconcave core $\widetilde U$ of $U$ by
$$\widetilde U=\bigcup_{S\in\Cc_p(U)}\supp(S).$$
\end{definition}

\begin{remark}\label{rk hull}
If  $U$ is an open subset of $\Pb^k$ such that $\Cc_{s+1}(\Pb^k\setminus U)\neq\varnothing$ then we can define the $(p-1)$-pseudoconvex hull $\widehat U$ of $U$ to be the complement of $\cup_{S\in\Cc_{s+1}(\Pb^k\setminus U)}\supp(S).$ Both this set and $\widetilde U$ have interesting geometric properties. When $s=0,$ $\widehat U$ is the rationally convex hull of $U,$ cf. \cite{guedj-appro} and \cite{duval-sibony-conv}. To our knowledge, these sets have not been studied in general.
\end{remark}

\subsection{Structural varieties}
We are now interested in the geometry of $\Cc_p(U).$ We will study this set by the means of special families parametrized by complex manifolds called structural varieties. The theory of structural varieties was introduced by Dinh and Sibony, cf. \cite{ds-geo}, \cite{d-attractor} and \cite{ds-superpot}. We refer to these references for more details on this concept. Here we just recall the facts that we will need in what follows.

Let $M$ and $X$ be two complex manifolds of dimension $m$ and $k$ respectively. Let $\pi_M\colon M\times X\to M$ and $\pi_X\colon M\times X\to X$ be the canonical projections. To a positive closed $(p,p)$-current $\mathcal R$ in $M\times X$ such that $\pi_X(\supp(\mathcal R)\cap\pi_M^{-1}(M'))\Subset X$ for all $M'\Subset M,$ the slicing theory of Federer \cite{federer} associates for almost all $\theta\in M$ a positive closed $(p+m,p+m)$-current $\langle\mathcal R,\pi_M,\theta\rangle.$ It is supported on $\{\theta\}\times X$ and we will identify it to a $(p,p)$-current of $X.$ The family $\mathcal R(\theta):=\langle\mathcal R,\pi_M,\theta\rangle$ defines a \textit{structural variety} parametrized by $M.$ We denote it by $\{\mathcal R(\theta)\}_{\theta\in M}$ or simply $\mathcal R.$ When $M$ is biholomorphic to a disk we call it \textit{a structural disk}. When $X$ is weakly $p$-pseudoconvex, a consequence of Lemma \ref{le sepa} is that $\mathcal R(\theta)$ is defined for all $\theta\in M.$ Another important point about structural varieties is that the function $\theta\mapsto\langle\mathcal R(\theta),\phi\rangle$ inherits properties from the test form $\phi.$

\begin{theorem}\label{th sh}\cite[Theorem 2.1]{ds-geo}\cite[Proposition A.1]{d-attractor}
Let $U$ be an open subset of $\Pb^k$ such that $\overline U$ is weakly $p$-pseudoconvex. Every structural variety in $\Cc_p(U)$ is defined everywhere. Moreover, if $\mathcal R$ is such a structural variety and if $\phi$ is a real continuous $(s,s)$-form on $U$ such that $\ddc\phi\geq0$ (resp. $\ddc\phi=0,$ resp. $d\phi=0$) then the function $h(\theta):=\langle\mathcal R(\theta),\phi\rangle$ is plurisubharmonic (resp. pluriharmonic, resp. constant). In particular, the mass of $\mathcal R(\theta)$ is independent of $\theta.$
\end{theorem}

Our approach deeply exploits the fact that the ambient space is $\Pb^k.$ In particular, that $\Pb^k$ has a big automorphism group which acts transitively on $\Pb^k$ but also on $T^*\Pb^k$ is important. This allows us to use convolution in order to construct structural disks which regularize currents (see e.g. \cite{ds-superpot}). To this aim, we introduce some notations.

We endow $\Pb^k$ with the distance induced by the Fubini-Study metric and if $\eta>0$ and $E\subset\Pb^k$ then we denote by $E_\eta$ the neighborhood of $E$ defined by $E_\eta:=\{x\in\Pb^k\,|\, \dist(x,E)<\eta\}.$ Let $W$ be a small neighborhood of $\id$ in $\aut(\Pb^k)$ which is biholomorphic to the unit ball $\Bb(0,1)$ of $\Cb^{k^2+2k}.$ We will often identify an automorphism $\sigma$ in $W$ with the corresponding point in $\Bb(0,1)$ and we assume that $\id$ is associated to $0.$ If $0<r\leq1,$ we denote by $B_W(r)$ the subset of $W$ corresponding to $\Bb(0,r).$ Moreover, we choose a smooth probability measure $\rho$ on $W$ with full support. We state without proof the following elementary lemma in order give notations.\begin{lemma}\label{le transi}
For all $0<r\leq1$ there exists $\eta(r)>0$ such that if $x,y\in\Pb^k$ with $\dist(x,y)\leq\eta(r)$ then there is $\sigma\in B_W(r)$ with $\sigma(x)=y.$ Moreover, if $K$ is a compact subset of an open set $V$ of $\Pb^k$ then there exists a constant $0<r\leq1$ such that for all $\sigma$ in $B_W(r),$ $\sigma(K)\subset V.$
\end{lemma}
The following construction will be repeatedly used in the sequel.
\begin{proposition}\label{prop regu}
Let $K$ be a compact subset of an open set $V$ of $\Pb^k.$ Let $0<r\leq1$ be as in Lemma \ref{le transi}. There exist constants $l>0$ and $C>0$ with the followings properties. For each current $S\in\Cc_p(K)$ there is a structural disk $\mathcal S$ in $\Cc_p(V)$ such that
\begin{itemize}
\item $\mathcal S(0)=S$ and for $\theta\neq0,$ $\mathcal S(\theta)$ is smooth with $\|\mathcal S(\theta)\|_{\mathcal C^2}\leq C|\theta|^{-l},$
\item for each $L\Subset\Db^*,$ there exists $C_L>0$ such that $\|\mathcal S(\theta)-\mathcal S(\theta')\|_{\mathcal C^2}\leq C_L|\theta-\theta'|$ if $\theta,\theta'\in L,$
\item if $\theta\neq0$ then $\mathcal S(\theta)$ is strictly positive on $(\supp(S))_{\eta(2^{-1}r|\theta|)}.$
\end{itemize}
Moreover, if $\theta_0\neq0$ then $\mathcal S(\theta_0)$ depends continuously on $S$ with respect to the smooth topology for $\mathcal S(\theta_0)$ and the weak topology for $S.$
\end{proposition}
\begin{proof}
Using the identification between $W$ and $\Bb(0,1)$ we can define
$$\mathcal S(\theta):=\int_{B_W(1)}(r\theta\sigma)_*Sd\rho(\sigma).$$
Almost all the properties stated above about $\mathcal S(\theta)$ can be found in \cite{ds-superpot} \cite{d-attractor} or are classical consequences of convolution. The only point which is well-known but not proved explicitly in the literature is the strict positivity of $\mathcal S(\theta).$ It simply comes from the fact that if $u$ is a (strongly) positive element of $\bigwedge^{(p,p)}\Cb^k$ and $v$ is a weakly positive element of $\bigwedge^{(s,s)}\Cb^k,$ both non-zero, then $(\sigma_*u)\wedge v>0$ for almost all elements $\sigma\in Gl_k(\Cb).$ Therefore, if $R$ is a current of the form $\delta_a\wedge\Psi$ where $\Psi\in\bigwedge^{(p,p)}T_a^*\Pb^k$ is (strongly) positive and non-zero then $\mathcal R(\theta):=\int_{B_W(1)}(r\theta\sigma)_*Rd\rho(\sigma)$ is strictly positive on $B(a,\eta(2^{-1}r|\theta|)).$ Thus the result follows since $S$ can be disintegrate into currents with support at a point.
\end{proof}

The two following results are basic but we state them explicitly since they play a central role in the sequel. They have been already used in \cite{d-attractor}.
\begin{lemma}\label{le domi-cons1}
Let $(u_n)_{n\geq0}$ be a uniformly bounded sequence of subharmonic functions defined on $\Db$ which is locally equicontinuous on $\Db^*.$ Assume that there exists $c\in\Rb$ such that
$$\limsup_{n\to\infty}u_n(\theta)\leq c,$$
for all $\theta\in\Db$ and
$$\lim_{n\to\infty}u_n(0)=c.$$
Then the sequence $(u_n)_{n\geq0}$ converges pointwise to the constant function $c.$
\end{lemma}
\begin{proof}
Since the sequence is uniformly bounded, there exists a subsequence $(u_{n_i})_{i\geq0}$ which converges in $L^1$ to a subharmonic function $u_\infty$ such that, for all $\theta\in\Db,$ $u_\infty(\theta)\geq\limsup_{i\to\infty}u_{n_i}(\theta)$ with equality outside a polar subset of $\Db.$ The local equicontinuity on $\Db^*$ implies that the convergence is pointwise on $\Db^*.$ Hence, $u_\infty\leq c$ on $\Db^*$ and the maximum principle gives that $u_\infty(\theta)=c$ for all $\theta\in\Db$ since $\lim_{n\to\infty}u_n(0)=c.$ The result follows since $u_\infty$ was an arbitrary limit value of $(u_n)_{n\geq0}.$
\end{proof}

\begin{lemma}\label{le domi-cons2}
Let $(a_n)_{n\geq0}$ and $(b_n)_{n\geq0}$ be two sequences of real numbers. If there are two constants $c\in\Rb$ and $\alpha>0$ such that
$$\limsup_{n\to\infty}b_n\leq\alpha c,\ \ \limsup_{n\to\infty}(a_n-b_n)\leq(1-\alpha)c\ \ \textrm{and}\ \ \lim_{n\to\infty}a_n=c$$
then $\lim_{n\to\infty}b_n=\alpha c.$ 
\end{lemma}
\begin{proof}
We have
\begin{align*}
\alpha c&\geq\limsup_{n\to\infty}b_n\geq\liminf_{n\to\infty}b_n=-\limsup_{n\to\infty}(-b_n)=\lim_{n\to\infty}a_n-\limsup_{n\to\infty}(a_n-b_n)\\
&\geq  c-(1-\alpha)c=\alpha c,
\end{align*}
which gives the desired result.
\end{proof}

\section{Dynamics in the set of currents}\label{sec courant}
This section is devoted to the proof of Theorem \ref{th equi as} and Theorem \ref{th finitude}. The main construction is contained in Theorem \ref{th conv1}. The idea, already present in \cite{d-attractor}, is that in the set of invariant currents an attracting current $\tau$ maximizes plurisubharmonic observables. However, as in our setting it may exist several attracting currents, this maximum cannot be global. One of the main difficulties is to associate to $\tau$ a trapping region on which this maximum is global and to prove that this trapping region is not to small (see Theorem \ref{th conv1} and Lemma \ref{le courant avec prop de cv}).

The organisation of this section is as follows. First, we briefly recall some results about holomorphic endomorphisms of $\Pb^k$ and we state Lemma \ref{le existence courant} which will allow us, in Section \ref{subsec dim}, to give equivalent formulations Daurat's definition of the dimension of an attracting set. We also fix the setting for the sequel and give the definition of attracting currents. Then, in Section \ref{subsec trap}, we explain how to associate to an invariant current $S$ a family of trapping regions $\Nc_S(r)$ which will be a key point in Lemma \ref{le courant avec prop de cv}. This brings us to Section \ref{subsec-ac} where we study equidistribution toward an attracting current. In particular, we prove Theorem \ref{th conv1} and Lemma \ref{le courant avec prop de cv}. They imply the existence and the finiteness of attracting currents which give Theorem \ref{th equi as} and the first point in Theorem \ref{th finitude}. The second part of this theorem is obtained in Section \ref{subsec Dp}. We conclude this section with results related to the speed of convergence.
\subsection{Holomorphic endomorphisms of $\Pb^k$}
We refer to \cite{ds-lec} for a detailed exposition on pluripotential methods in complex dynamics in several variables, in particular for the definition of the Green currents and push-forwards and pull-backs of currents.

We denote by $\Hc_d(\Pb^k)$ the space of holomorphic endomorphisms of $\Pb^k$ of algebraic degree $d.$ Let $f$ be in $\Hc_d(\Pb^k)$ with $d\geq2.$  If $l$ is an integer with $0\leq l\leq k$ then by Bézout theorem, the action of $f^*$ on the cohomology group $H^{l,l}(\Pb^k,\Rb)$ is the multiplication by $d^l.$ By duality, the action of $f_*$ on $H^{l,l}(\Pb^k,\Rb)$ is the multiplication by $d^{k-l}.$ Therefore, the normalized pull-back $d^{-l}f^*$ and the normalized push-forward $d^{l-k}f_*$ define operators from the set $\Cc_l(\Pb^k)$ to itself.

As we said in the introduction, there exists a special current $T^l$ in $\Cc_l(\Pb^k)$ called the \textit{Green $(l,l)$-current} of $f.$ It is invariant by $d^{-l}f^*$ and $d^{k-l}f_*$ and if $S$ is a smooth form in $\Cc_l(\Pb^k)$ then $d^{-ln}f^{n*}S$ converges to $T^l.$ Moreover, the Green $(1,1)$-current $T$ has continuous local potentials  so its self-intersection of order $l$ is well-defined and coincides with $T^l.$ The support $\Jc_l$ of $T^l$ is called the \textit{Julia set of order $l$} of $f.$ This defines a filtration of sets
$$\varnothing=:\Jc_{k+1}\subset\Jc_k\subset\cdots\subset\Jc_1\subset\Jc_0:=\Pb^k.$$
Dinh \cite{d-attractor} and de Thelin \cite{dethelin-selles} prove that these Julia sets are related to the topological entropy of $f$ restricted to compact subsets of $\Pb^k.$
\begin{theorem}\label{th entropy julia}
Let $1\leq l\leq k.$ If $K\subset \Pb^k$ is a compact set such that $K\cap\Jc_l=\varnothing$ then the topological entropy of $f$ restricted to $K$ is smaller or equal to $(l-1)\log d.$
\end{theorem}
The fact that the action of $f_*$ on $H^{l,l}(\Pb^k,\Rb)$ dominates the one on $H^{l+1,l+1}(\Pb^k,\Rb)$ will allow us to construct positive closed currents supported in a trapping region $U$ from positive currents, not necessarily closed, also supported on $U.$ This is a consequence of the following lemma which was establish by Dinh \cite[Proposition 4.7]{d-attractor}.

\begin{lemma}\label{le existence courant}
Let $f$ be in $\Hc_d(\Pb^k)$ with $d\geq2.$ Let $\chi$ be a positive smooth function in $\Pb^k.$ If $S$ is a current in $\Cc_{k-l}(\Pb^k)$ then the sequence $d^{-ln}(f^n)_*(\chi S)$ has bounded mass and each of its limit values is a positive closed $(k-l,k-l)$-current of $\Pb^k$ of mass $c:=\langle S\wedge T^l,\chi\rangle.$
\end{lemma}
On the other hand, the fact that the potential of $d^{-n}(f^n)^*\omega$ converges uniformly to the one of $T$ implies the following uniform continuity result, cf. \cite[Proposition 5.2]{d-attractor}. We recall that if $S\in\Cc_{k-l}(\Pb^k)$ and $\chi$ is a smooth function then we can define $(d^{-ln}f^n_*\chi S)\wedge T^l$ by $\langle(d^{-ln}f^n_*\chi S)\wedge T^l,\phi\rangle=\langle S\wedge T^l,\chi(f^{n*}\phi)\rangle.$ 
\begin{lemma}\label{le conti ts}
Let $S$ be a current in $\Cc_{k-l}(\Pb^k)$ and let $\chi$ be a positive smooth function. Assume that the sequence $(d^{-ln}(f^n)_*(\chi S))_{n\geq0}$ converges to $cS_\infty.$ If $(n_i)_{i\geq0}$ and $(m_i)_{i\geq0}$ are two sequences converging to $+\infty$ then
$$\lim_{i\to\infty}(d^{-ln_i}f^{n_i}_*\chi S)\wedge(d^{-lm_i}f^{m_i*}\omega^l)=\lim_{i\to\infty}(d^{-ln_i}f^{n_i}_*\chi S)\wedge T^l=cS_\infty\wedge T^l.$$
Moreover, the map from $\Cc_{k-l}(\Pb^k)$ to $\Cc_k(\Pb^k)$ defined by $S\mapsto S\wedge T^l$ is continuous.
\end{lemma}
\begin{proof}
First observe that by Lemma \ref{le existence courant}, $c=\langle S\wedge T^s,\chi\rangle$ and $S_\infty\in\Cc_{k-l}(\Pb^k).$ From this, the proof of the first point is identical to the one of \cite[Proposition 5.2]{d-attractor}. The second point is a classical consequence of the continuity of the potential of $T,$ see e.g. \cite{de-book} or \cite{ds-lec}.
\end{proof}

\subsection{Dimension of an attracting set and attracting currents}\label{subsec dim}
Daurat introduced the following definition in \cite{daurat}. Here, we still use the convention that $s=k-p.$
\begin{definition}[Daurat]\label{def-dim}
Let $A\subset\Pb^k$ be an attracting set. We say that $A$ has codimension $p$ or dimension $s$ if $\Cc_p(A)\neq\varnothing$ and $\Cc_{p-1}(A)=\varnothing.$ The dimension of a trapping region is by definition the dimension of the associated attracting set.
\end{definition}
As observe by Daurat, if $A$ has dimension $s$ then its Hausdorff dimension is larger or equal to $2s.$ The first results of this section give equivalent formulations of this definition and they imply Proposition \ref{prop-intro}. In particular, they show that an attracting set of codimension $p$ in $\Pb^k$ is weakly $p$-pseudoconvex which is a crucial point in our approach.
\begin{proposition}\label{prop def equiva}
Let $A\subset\Pb^k$ be an attracting set for $f.$ For $0\leq l\leq k$ the following properties are equivalent.
\begin{itemize}
\item[1)] $A\cap\Jc_l=\varnothing.$
\item[2)] $\Cc_l(\mathbb P^k\setminus A)\neq\varnothing.$
\item[3)] $\Cc_{k-l}(A)=\varnothing.$
\end{itemize}
\end{proposition}
\begin{proof}
As by definition $\Jc_l=\supp(T^l),$ if $A\cap\Jc_l=\varnothing$ then $T^l$ belongs to $\Cc_l(\mathbb P^k\setminus A)$ which is therefore not empty.

It is well-known by a cohomological argument that if $S$ is in $\Cc_l(\Pb^k)$ and $R\in\Cc_{k-l}(\Pb^k)$ then the support of $S$ intersects the one of $R.$ Therefore, if $\Cc_{l}(\Pb^k\setminus A)\neq\varnothing$ then the support of each current in $\Cc_{k-l}(\Pb^k)$ intersects $\Pb^k\setminus A.$ In particular $\Cc_{k-l}(A)=\varnothing.$

Finally, to prove that 3) implies 1), we proceed by contraposition. Let $U$ be a trapping region for $A$ and assume that $A\cap\Jc_l\neq\varnothing.$ Therefore, there exists a positive smooth function $\chi$ with compact support in $U$ such that $c:=\|\chi\omega^{k-l}\wedge T^l\|\neq 0.$ Hence, using Lemma \ref{le existence courant} and the fact that $f(U)\Subset U$ and $A=\cap_{n\geq0}f^n(U),$ each limit value of $d^{-ln}(f^n)_*(\chi \omega^{k-l})$ is of the form $cR$ with $R\in\Cc_{k-l}(A).$ In particular, $\Cc_{k-l}(A)$ is not empty.
\end{proof}
\begin{corollary}\label{coro p-pseudo}
Let $A\subset\Pb^k$ be an attracting set for $f$ with a trapping region $U.$ Then the following properties are equivalent.
\begin{itemize}
\item[1)] $A$ is of dimension $s.$
\item[2)] $A\cap \Jc_s\neq\varnothing$ and $A\cap \Jc_{s+1}=\varnothing.$
\item[3)] $\Cc_p(A)\neq\varnothing$ and $\Cc_{s+1}(\Pb^k\setminus A)\neq\varnothing.$
\end{itemize}
In particular, if $A$ is of dimension $s$ then $A$ and $\overline U$ are weakly $p$-pseudoconvex.
\end{corollary}
\begin{proof}
As $s=k-p,$ the first part is a direct consequence of Proposition \ref{prop def equiva}. This implies that if $A$ is of dimension $s$ then $T^{s+1}$ is in $\Cc_{s+1}(\Pb^k\setminus A).$ Therefore $T^{s+1}$ is also in $\Cc_{s+1}(\Pb^k\setminus\overline U)$ since it is invariant by $d^{-(p-1)}f_*.$ Hence $\overline U$ is weakly $p$-pseudoconvex.
\end{proof}
This result was already obtained in \cite{fw-attractor} for minimal quasi-attractors when $k=2.$ More precisely, they prove that a minimal quasi-attractor of $\Pb^k$ which is not finite must contain an entire curve.

In \cite{d-attractor}, Dinh considers trapping regions $U$ satisfying geometrical conditions that we denote by (HD) : there exist two linear subspaces $I$ and $L$ of dimension $p-1$ and $s$ respectively such that $I\cap U=\varnothing$ and $L\subset U.$ Moreover, for each $x\in L$ the unique dimension $p$ linear subspace $I(x)$ containing $I$ and $x$ has to intersect $U$ in a subset which is star-shaped with respect to $x$ in $I(x)\setminus I\simeq\Cb^p.$ In particular, $U$ has dimension $s.$ The key point in these assumptions (HD) is that for each  element $S$ of $\Cc_p(U)$ there exists a canonical structural disk which links $S$ to $[L].$ For such trapping regions, Dinh obtains almost all the results contained in this section. However, in a general setting it seems very difficult to investigate the structure of $\Cc_p(U).$ 

We now fix the setting for the rest of this section. From now on, $f$ is an element of $\Hc_d(\Pb^k)$ with $d\geq2.$ Since we are especially interested in the action of $f$ on $\Cc_p(\Pb^k),$ we define the following operators from $\Cc_p(\Pb^k)$ to itself. For $S$ is in $\Cc_p(\Pb^k),$ set
$$\Lambda S:=d^{-s}f_*S,$$
and for each $n\geq1$
$$\Delta_nS:=\frac{1}{n}\sum_{i=1}^n\Lambda^iS.$$
If $U$ is a trapping region of codimension $p$ then we consider the following two special subset of $\Cc_p(U).$ Define $\Ic_p(U)$ as the set of currents $S\in\Cc_p(U)$ such that $\Lambda S=S.$ Since $f(U)\Subset U,$ each limit value of $(\Delta_nS)_{n\geq1}$ with $S\in\Cc_p(U)$ belongs to $\Ic_p(U).$ Hence, $\Ic_p(U)$ is non-empty as $\Cc_p(A)\neq\varnothing.$ The set $\Dc_p(U)$ was introduced by Dinh in \cite{d-attractor}. It consists of all possible limit values of sequences of the form $(\Lambda^nS_n)_{n\geq0}$ with $S_n\in\Cc_p(U).$ As observed by Dinh, $S$ is in $\Dc_p(U)$ if and only if there exists a sequence $(S_n)_{n\geq0}$ in $\Dc_p(U)$ such that $S=S_0$ and $\Lambda S_{n+1}=S_n.$ We can now give the definition of an attracting current in $\Cc_p(U).$

\begin{definition}\label{def-attracting-current}
Let $U$ be a trapping region of codimension $p.$ We say that a current $\tau\in\Cc_p(U)$ is \emph{attractive} on $U$ if
$$\lim_{n\to\infty}\Delta_nR=\tau,$$
for all continuous form $R$ in $\Cc_p(U).$ In this situation, we say that $\tau$ is an \emph{attracting current}.
\end{definition}
Observe that if $\tau\in\Cc_p(U)$ is attractive on $U$ then it is the unique attracting current in $\Cc_p(U).$ We will see in the sequel that the convergence toward an attracting current is more general and that $R$ need not be closed nor positive. An important issue for what follows is to obtain a lower bound to the size of a region on which $\tau$ is attractive. To this purpose we consider the two following numbers, $r_U$ and $\eta_U,$ which are related to the rate of contraction of $f(U)\Subset U.$
\begin{definition}\label{def-taille}
Let $U\subset\Pb^k$ be a trapping region. We denote by $r_U$ the maximum of the number $0<r\leq1$ such that $f\circ\sigma(\overline{U})\subset U$ for all $\sigma\in B_W(r).$ And we set $\eta_U:=\eta(r_U),$ where $\eta(r_U)$ is defined in Lemma \ref{le transi}.
\end{definition}
An important observation is that $f(U_{\eta_U})\subset U.$ The number $r_U$ depends on $f$ and it increases if we exchange $f$ by an iterate.

\subsection{Trapping regions associated to an invariant current}\label{subsec trap}
In this subsection we will construct for each invariant current a family of trapping regions. The construction is elementary and only uses the fact that the support of an invariant current is also invariant. However, we will later see that these trapping regions are related to subsets of $\Cc_p(U)$ which are in some sense pathwise-connected and this fact will be essential in the sequel.

Let $K\subset\Pb^k$ be an invariant compact subset for $f,$ $f(K)=K.$ If $0<r\leq1$ then we define $\Nc_K(r)$ to be the set of pseudo-orbits obtained with $B_W(r)$ and starting on $K.$ To be more precise, a point $x$ belongs to $\Nc_K(r)$ if and only if there exist a finite sequence $\sigma_1,\ldots,\sigma_l$ in $B_W(r)$ and a point $y\in K$ such that
$$x=f\circ\sigma_1\circ f\circ\ldots\circ f\circ\sigma_l(y).$$
Notice that $\Nc_K(r)$ depends only on $r,$ $K,$ $W$ and $f.$ Obviously, since $f(K)=K$ and $\id\in B_W(r)$ we have $K\subset\Nc_K(r).$ As $f$ is an open map and $B_W(r)$ is open, it is easy to check that $\Nc_K(r)$ is open for all $0<r\leq1.$ Moreover, by definition $f\circ\sigma(\Nc_K(r))\subset\Nc_K(r)$ for all $\sigma\in B_W(r).$ Since $\overline{\Nc_K(r)}\subset\cup_{\sigma\in B_W(r)}\sigma(\Nc_K(r)),$ it follows that $f(\Nc_K(r))\Subset\Nc_K(r),$ i.e. $\Nc_K(r)$ is a trapping region. We can also observe that $\overline{\Nc_K(r')}\subset\Nc_K(r)$ for all $r'<r$ since $f$ is an open map and therefore
$$\Nc_K(r)=\bigcup_{0<r'<r}\overline{\Nc_K(r')}.$$
In particular, for each $0<r\leq1$ there exists $0<r'<r$ such that $f(\Nc_K(r))\Subset\Nc_K(r').$ Moreover, observe that $f(\Nc_K(r)_{\eta(r)})\subset\Nc_K(r)$ where $\eta(r)$ is defined in Lemma \ref{le transi} and is independent of $K.$

In general $\Nc_K(r)$ can be much bigger than $K.$ For example on $\Pb^1,$ if $K$ is a Siegel fixed point then $\Nc_K(r)=\Pb^1$ for all $r>0.$ In the same way, $\Nc_{\Jc_k}(r)=\Pb^k$ for all $r>0.$ However, if $K$ is an attracting set then the attracting set associated to $\Nc_K(r)$ is $K$ for $r>0$ small enough.

In the sequel, we are interested to the following setting. Let $U$ be a trapping region of codimension $p.$ Let $S$ be a current in $\Ic_p(U).$ The set $\supp(S)$ is invariant thus we can define for $0<r\leq1$
$$\Nc_S(r):=\Nc_{\supp(S)}(r).$$
Since $f\circ\sigma(U)\subset U$ for all $\sigma\in B_W(r_U),$ the set $\Nc_S(r)$ is contained in $U$ if $r\leq r_U.$ Moreover, $\supp(S)\subset\Nc_S(r)$ hence $\Cc_p(\Nc_S(r))\neq\varnothing.$ Therefore, $\Nc_S(r)$ is a trapping region of codimension $p$ if $r\leq r_U.$ We also consider the following subsets of $\Cc_p(\Pb^k).$ For $0<r<1,$ define $\overline{\Cc_S(r)}$ as the smallest closed convex set in $\Cc_p(\Pb^k)$ which contains all currents of the form
$$\Lambda\sigma_{1*}\Lambda\sigma_{2*}\cdots\Lambda\sigma_{l*}S,$$
where $\sigma_1,\ldots,\sigma_l$ are elements of $B_W(r).$ The discussion above implies that $\overline{\Cc_S(r)}\subset\Cc_p(U)$ when $r<r_U.$ Finally, for $0<r\leq1,$ we set $\Cc_S(r):=\cup_{r'<r}\overline{\Cc_S(r')}.$
\begin{proposition}
Let $0<r\leq1.$ The set $\Nc_S(r)$ satisfies
$$\Nc_S(r)=\bigcup_{R\in\Cc_S(r)}\supp(R).$$
\end{proposition}
\begin{proof}
If $x$ is in $\Nc_S(r)$ then there exist $\sigma_1,\ldots,\sigma_l$ in $B_W(r)$ and $y\in\supp(S)$ such that $x=f\circ\sigma_1\circ f\circ\ldots\circ f\circ\sigma_l(y).$ In particular, $x$ is in the support of $R:=\Lambda\sigma_{1*}\Lambda\sigma_{2*}\cdots\Lambda\sigma_{l*}S.$ But there exists $r'<r$ such that $\sigma_1,\ldots,\sigma_l$ are also in $B_W(r').$ Hence, $R$ is in $\overline{\Cc_S(r')}\subset\Cc_S(r)$ and $\Nc_S(r)\subset\cup_{R\in\Cc_S(r)}\supp(R).$

Conversely, a current of the form $R:=\Lambda\sigma_{1*}\Lambda\sigma_{2*}\cdots\Lambda\sigma_{l*}S$ with $\sigma_1,\ldots,\sigma_l$ in $B_W(r')$ is in $\Cc_p(\Nc_S(r'))$ and therefore $\overline{\Cc_S(r')}\subset\Cc_p(\overline{\Nc_S(r')}).$ The second inclusion follows since
$$\Nc_S(r)=\bigcup_{0<r'<r}\overline{\Nc_S(r')}\ \ \text{and}\ \ \Cc_S(r)=\bigcup_{r'<r}\overline{\Cc_S(r')}.$$
\end{proof}
As we are especially interested in invariant currents, we define $\overline{\Ic_S(r)}:=\overline{\Cc_S(r)}\cap\Ic_p(U).$ It is a compact convex set and it satisfies the following monotonic property.
\begin{lemma}\label{le monotonic}
Let $0<r<1$ and let $S$ be in $\Ic_p(U).$ If $R$ is in $\overline{\Ic_S(r)}$ then
$$\overline{\Cc_R(r)}\subset\overline{\Cc_S(r)}\ \ \text{and}\ \ \overline{\Ic_R(r)}\subset\overline{\Ic_S(r)}.$$
\end{lemma}
\begin{proof}
Observe that the second inclusion follows easily from the first one. Assume that $R\in\overline{\Ic_S(r)}$ and let $R'$ be in $\overline{\Cc_R(r)}.$ As $R$ is in $\overline{\Cc_S(r)},$ by definition there exists a sequence of probability measures $(\rho_n)_{n\geq1}$ where $\rho_n$ is defined on $B_W(r)^n$ such that
$$R_n:=\int_{B_W(r)^n} \Lambda\sigma_{1*}\cdots\Lambda\sigma_{n*}Sd\rho_n(\sigma_1,\ldots,\sigma_n)$$
converge toward $R.$ In the same way, there exists a sequence of probability measures $(\rho'_n)_{n\geq1}$ such that
$$R'_n:=\int_{B_W(r)^n} \Lambda\sigma_{1*}\cdots\Lambda\sigma_{n*}Rd\rho'_n(\sigma_1,\ldots,\sigma_n)$$
converge to $R'.$ Therefore, it is easy to see that there exists an increasing sequence $(\psi(n))_{n\geq1}$ such that the sequence of currents defined by
$$\int_{B_W(r)^n} \Lambda\sigma'_{1*}\cdots\Lambda\sigma'_{n*}\left(\int_{B_W(r)^{\psi(n)}} \Lambda\sigma_{1*}\cdots\Lambda\sigma_{\psi(n)*}Sd\rho_{\psi(n)}(\sigma_1,\ldots,\sigma_{\psi(n)})\right)d\rho'_n(\sigma_1,\ldots,\sigma_n),$$
converge to $R'.$ Hence, $R'$ is in $\overline{\Cc_S(r)}.$
\end{proof}
A current $R$ in $\overline{\Cc_S(r)}$ is not necessary linked to $S$ by a structural disk a priori. However, we can approximate $R$ by currents linked to $S$ with good properties.
\begin{lemma}\label{le disque dans Cc_S}
Let $0<r<1$ and let $R$ be in $\overline{\Cc_S(r)}.$ There exists a sequence of structural disks $(\mathcal R_n)_{n\geq1}$ in $\Cc_p(\Pb^k)$ such that
\begin{itemize}
\item $\mathcal R_n(0)=S$ for all $n\geq1,$
\item $\lim_{n\to\infty}\mathcal R_n(r)=R,$
\item $\mathcal R_n(\theta)$ belongs to $\overline{\Cc_S(|\theta|)}$ and in particular $\mathcal R_n(\theta)\in\Cc_p(U)$ if $|\theta|<r_U,$
\item if $\theta\neq0$ then $\mathcal R_n(\theta)$ is the image by $\Lambda$ of a smooth current.
\end{itemize}
\end{lemma}
\begin{proof}
Let $0<r<1.$ Let $R$ be in $\overline{\Cc_S(r)}.$ As in the proof of Lemma \ref{le monotonic}, there exists a sequence of probability measures $(\rho_n)_{n\geq1}$ such that
$$R_n:=\int_{B_{W}(r)^n} \Lambda\sigma_{1*}\cdots\Lambda\sigma_{n*}Sd\rho_n(\sigma_1,\ldots,\sigma_n)$$
converge toward $R.$ Using the fact that $W$ is biholomorphic to the unit ball $\Bb(0,1)$ in $\Cb^{k^2+2k}$ and by identifying $\sigma\in W$ with the corresponding point in $\Bb(0,1),$ we can define the structural disk
$$R_n(\theta):=\int_{B_W(r)^n} \Lambda(\widetilde\sigma_1(\theta))_*\cdots\Lambda(\widetilde\sigma_n(\theta))_*Sd\rho_n(\sigma_1,\ldots,\sigma_n).$$
Here $\widetilde\sigma(\theta):=\theta\sigma/r$ if $\sigma\in B_W(r).$ This disk satisfies $R_n(0)=S$ since $\id\in W$ corresponds to $0\in\Bb(0,1)$ and $\Lambda S=S.$ Moreover, $\theta\sigma/r\in B_W(|\theta|)$ if $\sigma\in B_W(r)$ and thus $R_n(\theta)\in\overline{\Cc_S(|\theta|)}.$ Therefore, the sequence $(R_n(\theta))_{n\geq1}$ satisfies the first three points in the lemma. It is enough to modify it slightly in order to obtain the last point. Indeed, arguing as in Proposition \ref{prop regu}, it is easy to check that if $\rho$ is a smooth probability measure on $W$ and $(a_n)_{n\geq1}$ is an increasing sequence of positive numbers converging to $1$ sufficiently fast then
$$\mathcal R_n(\theta):=\int_W\int_{B_W(r)^n} \Lambda(a_n\widetilde\sigma_1(\theta)+(1-a_n)\theta\sigma)_*\cdots\Lambda(\widetilde\sigma_n(\theta))_*Sd\rho_n(\sigma_1,\ldots,\sigma_n)d\rho(\sigma)$$
fulfils all the four points.
\end{proof}

\subsection{Existence and finiteness of attracting currents}\label{subsec-ac}
The next theorem gives the existence of attracting currents and it implies Theorem \ref{th equi as}. However, it is slightly more precise with, in particular, an explicit estimate on the size of the trapping region $D_\tau.$ This last point constitutes the keystone for the finiteness results in the sequel.
\begin{theorem}\label{th conv1}
Let $U$ be a codimension $p$ trapping region. There exists an attracting current $\tau$ which is extremal in $\Ic_p(U).$ Moreover, $\tau$ is attractive on a trapping region $D_\tau$ such that $f(\sigma(D_\tau))\subset D_\tau$ for all $\sigma\in B_W(r_U).$ In particular, $\Nc_\tau(r_U)\subset D_\tau.$
\end{theorem}
In order to prove the theorem, let $(\phi_j)_{j\geq1}$ be a dense sequence in $\cali P(U).$ By Corollary \ref{coro p-pseudo} and Lemma \ref{le sepa}, two elements $R$ and $S$ of $\Cc_p(U)$ are equal if and only if $\langle R,\phi_j\rangle=\langle S,\phi_j\rangle$ for all $j\geq1.$ We will use this sequence $(\phi_j)_{j\geq1}$ to construct a decreasing sequence of trapping regions in $U$ which will turn out to be stationary. The limit of this stationary sequence will be $D_\tau.$

In this purpose, we define inductively $D_0:=U,$ $M_0:=\Cc_p(U)$ and for $j\geq1$
$$c_j:=\max_{S\in\Ic_p(D_{j-1})}\langle S,\phi_j\rangle,$$
$$M_j:=\{S\in\Cc_p(D_{j-1})\,|\, \langle\Delta_nS,\phi_j\rangle\to c_j\},$$
and
$$D_j:=\cup_{S\in M_j}\supp(S).$$
We will see in the proof of the following proposition that these objects are well-defined.
\begin{proposition}\label{prop conv1}
Each $D_j$ is a non-empty open subset of $\Pb^k.$ Moreover, $f(\sigma(D_j))\subset D_j$ for all $\sigma$ in $B_W(r_U)$ and $j\geq0.$ In particular, $f(D_{j,\eta_U})\Subset D_j.$ Here $\eta_U$ is defined in Definition \ref{def-taille}.
\end{proposition}
\begin{remark}
The fact that $\eta_U$ doesn't depend on $j$ is important for the sequel. It gives in some sense a uniform rate of contraction.
\end{remark}
\begin{proof}
Since by definition $f(\sigma(U))\subset U$ for all $\sigma$ in $B_W(r_U),$ the proposition holds for $j=0.$

Now let $j\geq1$ and assume that the proposition holds for $D_{j-1}.$ In particular $D_{j-1}$ is non-empty and thus by definition $\Cc_p(D_{j-1})\neq\varnothing.$ As $D_{j-1}$ is invariant, it implies that $\Ic_p(D_{j-1})$ is a non-empty compact set and therefore, there exists $R\in\Ic_p(D_{j-1})$ such that $\langle R,\phi_j\rangle=\max_{S\in\Ic_p(D_{j-1})}\langle S,\phi_j\rangle=:c_j.$ It follows that $R$ belongs to $M_j$ and thus $D_j$ is not empty.

The following arguments will be recurrent in the sequel. We explain them in detail here in order to be able to be more succinct in the later uses. Let $R$ be a current of $M_j.$ Since $D_{j-1}$ is open, we can consider the structural disk $\mathcal R$ of $\Cc_p(D_{j-1})$ obtained from $R$ by Proposition \ref{prop regu} with $V=D_{j-1}$ and $K=\supp(R).$ By Theorem \ref{th sh}, the functions
$$u_n(\theta):=\langle \Delta_n\mathcal R(\theta),\phi_j\rangle$$
are subharmonic. They satisfy $0\leq u_n\leq1$ since $0\leq\phi_j\leq\omega^s.$ Moreover, Proposition \ref{prop regu} implies that they are locally equicontinuous on $\Db^*.$ By definition of $M_j$ and $c_j,$ we have
$$\lim_{n\to\infty}u_n(0)=c_j\geq\limsup_{n\to\infty}u_n(\theta)$$
for all $\theta\in\Db.$ Therefore, by Lemma \ref{le domi-cons1}, the sequence $(u_n)_{n\geq1}$ converges pointwise to $c_j.$ Hence $\mathcal R(\theta)\in M_j$ for all $\theta\in\Db.$ Therefore, $D_j$ is also open since $\supp(\mathcal R(\theta))$ is a neighborhood of $\supp(R)$ when $\theta\neq0.$

Finally, in order to prove that $D_j$ is a trapping region, first observe that $R\in\Cc_p(D_{j-1})$ belongs to $M_j$ if and only if $\Lambda R$ is also in $M_j.$ Indeed, since $f(D_{j-1})\Subset D_{j-1},$ $\Lambda R\in\Cc_p(D_{j-1})$ and the fact that each limit value of $\Delta_nR$ is $\Lambda$-invariant implies that the limit values of $\langle\Delta_nR,\phi_j\rangle$ are the same than those of $\langle\Delta_n(\Lambda R),\phi_j\rangle.$ Now, let $x\in D_j$ and $\sigma\in B_W(r_U).$ By Proposition \ref{prop regu} applied with $V=f^{-1}(U),$ $K=\overline U$ and $r=r_U,$ there exists a structural disk $\mathcal R$ of $\Cc_p(f^{-1}(U))$ such that $\supp(\mathcal R(\theta))$ is the union over all $\widetilde\sigma\in B_W(|\theta|r_U)$ of $\widetilde\sigma(\supp(R)).$ In particular, $\sigma(x)$ belongs to $\supp(\mathcal R(\theta))$ for $|\theta|<1$ large enough. On the other hand, by the induction hypothesis, $\{\Lambda \mathcal R(\theta)\}_{\theta\in\Db}$ is a structural disk in $\Cc_p(D_{j-1})$ with $\Lambda \mathcal R(0)=\Lambda R$ in $M_j.$ As above, by considering the functions
$$v_n(\theta):=\langle \Delta_n\Lambda \mathcal R(\theta),\phi_j\rangle$$
which are also locally equicontinuous on $\Db^*,$ we obtain that for all $\theta\in\Db,$ $\Lambda \mathcal R(\theta)$ is in $M_j$ and therefore $f(\sigma(x))$ is in $D_j.$
\end{proof}
So each $D_j$ is a trapping region such that $D_{j+1}\subset D_j.$ The associated attracting sets $E_j:=\cap_{n\geq1}f^n(D_j)$ therefore also define a decreasing sequence of sets.
\begin{remark}
If we choose $\phi_1=\omega^s$ then all $S\in\Cc_p(U)$ satisfy $\langle S,\phi_1\rangle=1$ and $M_1$ is equal to $\Cc_p(U).$ Therefore, $D_1$ is equal to the union of the support of all currents in $\Cc_p(U),$ i.e. it is the $s$-pseudoconcave core $\widetilde U$ of $U$ and thus $\widetilde U$ is a trapping region. The attracting set associated to it is in some sense the part of ``pure'' dimension $s$ in $A.$ Notice that it is easy to show that the $(p-1)$-pseudoconvex hull $\widehat U$ defined in Remark \ref{rk hull} is also a codimension $p$ trapping region.
\end{remark}
The next step toward Theorem \ref{th conv1} is to show that these sequences are both stationary. For the sets $(E_j)_{j\geq0},$ it will simply comes from the uniform rate of contraction, $f(D_{j,\eta_U})\Subset D_j$ with $\eta_U>0$ independent of $j.$
\begin{lemma}\label{le stat1}
Let $\eta>0.$ There exists a constant $m\geq1$ which depends only on $\eta$ such that each monotonic sequence of attracting sets $(A_j)_{j\geq1}$ in $\Pb^k$ admitting a trapping region $V_j$ with $f(V_{j,\eta})\Subset V_j$ has at most $m$ distinct elements.
\end{lemma}
\begin{proof}
Let $i,j\geq1.$ If $A_i$ is not included in $A_j,$ it is neither included in $V_{j,\eta}.$ Therefore, there exists $x_i\in A_i$ such that $B(x_i,\eta)\cap V_j=\varnothing.$ It easily follows that $B(x_i,\eta/3)\subset V_{i,\eta/2}\setminus V_{j,\eta/2}.$ Since $\Pb^k$ has volume $1,$ it cannot exist more than $m:=Vol(B(x,\eta/3))^{-1}$ such distinct balls.
\end{proof}
We need the following proposition to prove that the sequence of trapping regions is also stationary.
\begin{proposition}\label{prop conv continu}
The set $M_j$ contains all continuous elements of $\Cc_p(D_j).$
\end{proposition}
\begin{proof}
Let $R$ be a continuous form in $\Cc_p(D_j).$ Since $\supp(R)$ is a compact subset of $D_j$ and $D_j=\cup_{S\in M_j}\supp(S),$ by Proposition \ref{prop regu} there exist currents $S_i$ in $M_j,$ $1\leq i\leq n_0,$ and structural disks $\mathcal S_i$ in $\Cc_p(D_{j})$ such that $\mathcal S_i(0)=S_i$ and $R\leq C\sum_{i=1}^{n_0} \mathcal S_i(\theta_i)$ for some $C>0$ and $\theta_1,\ldots,\theta_{n_0}\in\Db.$ By the definition of $c_j$ we have that
$$\limsup_{n\to\infty}\left\langle\Delta_n\left(C\sum_{i=1}^{n_0} \mathcal S_i(\theta_i)-R\right),\phi_j\right\rangle\leq (Cn_0-1)c_j$$
and
$$\limsup_{n\to\infty}\left\langle\Delta_n R,\phi_j\right\rangle\leq c_j.$$
But, as we have seen, $S_i\in M_j$ implies that $\mathcal S_i(\theta_i)\in M_j$ and then
$$\lim_{n\to\infty}\left\langle\Delta_n \mathcal S_i(\theta_i),\phi_j\right\rangle= c_j.$$
Therefore, by Lemma \ref{le domi-cons2} $\lim_{n\to\infty}\left\langle\Delta_n R,\phi_j\right\rangle=c_j,$ i.e. $R\in M_j.$
\end{proof}
As a consequence, $D_j$ is the union of the support of continuous forms in $\Cc_p(D_j).$ It also gives the following result.
\begin{proposition}\label{prop stat2}
If $E_{j+1}=E_j$ then $D_{j+1}=D_j.$
\end{proposition}
\begin{proof}
We already know that $D_{j+1}\subset D_j.$ If $E_{j+1}=E_j$ then there exists $N>0$ such that $f^N(D_j)\Subset D_{j+1}.$ As above, since $\overline{f^N(D_j)}$ is a compact subset of $D_{j+1},$ there exists a smooth form $S\in\Cc_p(D_{j+1})$ whose support contains $f^N(D_j).$ In particular, by Proposition \ref{prop conv continu} $S$ belongs to $M_{j+1}.$ Define
$$\widetilde S:=\frac{1}{d^{Np}}(f^N)^*S.$$
It is a current in $\Cc_p(f^{-N}(D_{j+1}))$ such that $\lim_{n\to\infty}\langle\Delta_n\widetilde S,\phi_{j+1}\rangle=c_j$ since $f_*f^*$ is equal to $d^k$ times the identity on positive closed currents. The fact that $f^N(D_j)$ is contained in $\supp(S)$ implies that $D_j\subset \supp(\widetilde S).$ Therefore, if $\widetilde{\mathcal S}$ is the regularization given by Proposition \ref{prop regu} with $V=f^{-N}(D_{j+1})$ then it satisfies
\begin{itemize}
\item $\widetilde {\mathcal S}(\theta)>0$ on $\overline{D_j}$ if $\theta\neq0,$
\item $\Lambda^N\widetilde {\mathcal S}(\theta)\in\Cc_p(D_{j+1})$ for all $\theta\in\Db,$
\item $\Lambda^N\widetilde {\mathcal S}(0)=S\in M_{j+1}.$
\end{itemize}
Hence, we obtain exactly as in the proof of Proposition \ref{prop conv continu} that all continuous forms in $\Cc_p(D_j)$ are in $M_{j+1}.$ That implies the desired result.
\end{proof}
We deduce form this our first result of convergence.
\begin{proof}[Proof of Theorem \ref{th conv1}]
By Lemma \ref{le stat1} and Proposition \ref{prop stat2} there is $j_0\geq1$ such that $E_j=E_{j_0}$ and $D_j=D_{j_0}$ for all $j\geq j_0.$ As $M_j$ contains all continuous elements of $\Cc_p(D_j)$ and that $\Cc_p(D_{j_0})\subset\Cc_p(D_j)$ for all $j\geq0,$ it follows that
$$\lim_{n\to\infty}\langle\Delta_nS,\phi_j\rangle=c_j,$$
for all $j\geq0$ and for all continuous forms $S\in\Cc_p(D_{j_0}).$ Thus, each limit value $S_\infty$ of $\Delta_nS$ must satisfy
$$\langle S_\infty,\phi_j\rangle=c_j,$$
which, by Lemma \ref{le sepa}, completely determines $S_\infty.$ In particular, this limit value is unique and doesn't depend on $S.$ We call it $\tau$ and set $D_\tau:=D_{j_0}.$ It is an invariant current and $\Nc_\tau(r_U)\subset D_\tau$ since $\tau\in\Cc_p(D_\tau)$ and $f\circ\sigma(D_\tau)\subset D_\tau$ for all $\sigma\in B_W(r_U).$

To prove that $\tau$ is extremal in $\Ic_p(U)$ let $\tau_1$ and $\tau_2$ be two elements of $\Ic_p(U)$ such that $\tau=2^{-1}(\tau_1+\tau_2).$ Define $i_0=\min\{i\geq1\,|\, \langle\tau,\phi_i\rangle\neq\langle \tau_1,\phi_i\rangle\}.$ If $i_0\neq+\infty$ then the construction of $\tau$ implies that $\langle \tau,\phi_{i_0}\rangle>\langle\tau_1,\phi_{i_0}\rangle$ and $\langle\tau,\phi_{i_0}\rangle\geq\langle\tau_2,\phi_{i_0}\rangle$ which is impossible. Therefore, $i_0=+\infty$ and $\tau=\tau_1=\tau_2.$
\end{proof}
\begin{remark}
The set $\Nc_\tau(r_U)$ is contained in $D_\tau$ but they are not equal in general. For an attractive fixed point $z_0,$ $\tau$ is the Dirac mass at $z_0,$ $\Nc_\tau(r_U)$ is included in the immediate basin of $z_0$ while $D_\tau$ equal to the whole trapping region $U.$
\end{remark}
The following remark will be crucial in the sequel.
\begin{remark}\label{rk max}
As a consequence of Theorem \ref{th conv1}, if $S$ is an attracting current of bidegree $(p,p)$ which is attractive on a trapping region $V$ then
$$\langle S,\phi\rangle=\max_{R\in\Ic_p(V)}\langle R,\phi\rangle,$$
for all $\phi\in\Pc(V).$ Indeed, it simply comes from the facts that $S$ is the unique attracting current in $\Cc_p(V)$ and that, for each $\phi\in\Pc(V),$ the construction in the proof of Theorem \ref{th conv1} gives an attracting current $\tau$ such that $\langle \tau,\phi\rangle=\max_{R\in\Ic_p(V)}\langle R,\phi\rangle$ if we start with a dense sequence $(\phi_j)_{j\geq1}$ in $\Pc(V)$ such that $\phi_1=\phi.$ In particular,
$$\langle S,\phi\rangle=\langle R,\phi\rangle$$
for all $R\in\Ic_p(V)$ and all smooth form $\phi$ such that $\ddc\phi=0.$
\end{remark}

Another direct consequence of Theorem \ref{th conv1} is that attracting currents are extremal.
\begin{corollary}
Attracting currents of bidegree $(p,p)$ in $U$ are extremal points of $\Ic_p(U).$ In particular, if such a current $S$ puts mass on an analytic set of dimension $s$ then $S$ is a combination of currents of integration on analytic sets of pure dimension $s.$
\end{corollary}
\begin{proof}
Let $S$ be an attracting current in $\Cc_p(U).$ By definition, $S$ is attractive on some trapping region $V.$ If $S$ is decomposable in $\Ic_p(U),$ $S=2^{-1}(S_1+S_2),$ then $S_1$ and $S_2$ are supported in $V,$ i.e. $S$ is also decomposable in $\Ic_p(V).$ On the other hand, by Theorem \ref{th conv1} applied to $V,$ there exists an attracting current which is extremal in $\Ic_p(V).$ But this implies that $S=S_1=S_2$ since $S$ is the unique attracting current in $\Ic_p(V).$

For the second point, by Siu's decomposition theorem, there exist analytic sets $[H_i]$ of pure dimension $s,$ positive numbers $c_i,$ and a positive closed current $S'$ having no mass on analytic sets of dimension $s$ such that $S=\sum c_i[H_i]+S'.$ Moreover, this decomposition is stable under $\Lambda$ i.e. $\Lambda(\sum c_i[H_i])$ is a combination of currents of integration and $\Lambda S'$ has no mass on analytic set of dimension $s.$ The fact that $\Lambda S=S$ implies that $\Lambda(\sum c_i[H_i])=\sum c_i[H_i]$ and $\Lambda S'=S'.$ Since $S$ is extremal in $\Ic_p(U),$ one of these two currents must vanish.
\end{proof}
\begin{remark}
1) There exist examples of attracting currents which are not supported by a pluripolar set but with non-zero Lelong number at some points, cf. \cite[Section 6.3]{daurat}.\\
2) It would be interesting to know if the fact that an attracting current is algebraic implies that the associated attracting set is algebraic.
\end{remark}

In the remaining part of this section, we will extend little by little our results about equidistribution. The key idea is simple. We will use structural disks of center $S\in\Ic_p(U)$ to dominate a current or a part of it and then deduce from Theorem \ref{th conv1} information about the possible limit values when we apply the dynamics. A good example of this strategy will be the proof of Lemma \ref{le courant avec prop de cv}. But, first we need to establish several intermediate results which give information about attracting currents. The following one says that if $S$ is attractive on $V$ then it is also attractive on the full basin of $V$ and not only for closed forms. This last point is the counterpart in our setting of a result obtained by Dinh in \cite[Section 4]{d-attractor}.

\begin{lemma}\label{le conv tronca}
Let $S$ be an attracting current of bidegree $(p,p)$ which is attractive on a trapping region $V.$ Let $B_V:=\cup_{n\geq0}f^{-n}V$ be the basin of $V.$ For all continuous forms $R$ in $\Cc_p(\Pb^k)$ and all positive smooth functions $\chi$ with compact support in $B_V$ we have that
$$\lim_{n\to\infty}\Delta_n(\chi R)=cS,$$
where $c=\langle R\wedge T^s,\chi\rangle.$
\end{lemma}
\begin{proof}
By Lemma \ref{le existence courant} there exists $R_\infty\in\Cc_p(V)$ and an increasing sequence $(n_i)_{i\geq0}$ such that $\lim_{i\to\infty}\Lambda^{n_i}\chi R=cR_\infty.$ If $c=0$ then the lemma follows easily. Assume that $c>0$ and choose $c'>0$ such that $c'<c.$ Let $\mathcal R_\infty$ be the regularization of $R_\infty$ in $\Cc_p(V)$ given by Proposition \ref{prop regu}. There exist a smooth positive function $\chi_\infty\leq1$ and $\theta_0\in\Db^*$ such that $\chi_\infty$ is equal to $1$ on $\supp(R_\infty)$ and $\supp(\chi_\infty)\subset\supp(\mathcal R_\infty(\theta_0)).$ In particular, $\langle R_\infty\wedge T^s,\chi_\infty\rangle=1.$ Hence, there exists $i_0\geq 0$ such that $c'':=\langle(\Lambda^{n_{i_0}}\chi R)\wedge T^s,\chi_\infty\rangle\geq c'$ and thus $\langle R\wedge T^s,\chi(\chi_\infty\circ f^{n_{i_0}})\rangle=c''\geq c'.$

Since $\lim_{N\to\infty}\Delta_N\mathcal R_\infty(\theta_0)=S,$ exactly as in Proposition \ref{prop stat2} we can use a regularization in $\Cc_p(f^{-n_{i_0}}(V))$ of the smooth current
$$\widetilde R_\infty:=\frac{1}{d^{pn_{i_0}}}(f^{n_{i_0}})^*\mathcal R_\infty(\theta_0),$$
in order to obtain a structural disk $\widetilde{\mathcal R}_\infty$ such that
\begin{itemize}
\item $\|\widetilde {\mathcal R}_\infty(\theta)-\widetilde {\mathcal R}_\infty(\theta')\|_{\mathcal C^0}\leq C|\theta-\theta'|,$
\item $\widetilde {\mathcal R}_\infty(\theta)>0$ on $\supp(\chi_\infty\circ f^{n_{i_0}}(\chi R))$ if $\theta\neq0,$
\item $\Lambda^{n_{i_0}}\widetilde {\mathcal R}_\infty(\theta)\in\Cc_p(V)$ for all $\theta\in\Db,$
\item $\lim_{N\to\infty}\Delta_N\widetilde {\mathcal R}_\infty(0)=S.$
\end{itemize}

On the other hand, since $S$ is the unique attracting current in $\Cc_p(V),$ it satisfies $\langle S,\phi\rangle=\max_{S'\in\Ic_p(V)}\langle S',\phi\rangle$ for all $\phi\in\Pc(V).$ Therefore,
\begin{align*}
\limsup_{N\to\infty}\langle\Delta_N\widetilde {\mathcal R}_\infty(\theta),\phi\rangle\leq\langle S,\phi\rangle,
\end{align*}
for all $\phi\in\Pc(V).$ And exactly as in the proof of Proposition \ref{prop conv1}, the equality for $\theta=0$ implies that $\lim_{N\to\infty}\Delta_N\widetilde {\mathcal R}_\infty(\theta)=S$ for all $\theta\in\Db.$ In a similar way, if $\theta_1$ is in $\Db^*$ then there exists $0<c_1<1$ such that 
$$\widetilde {\mathcal R}_\infty(\theta_1)\geq c_1(\chi_\infty\circ f^{n_{i_0}}(\chi R))$$
and therefore
\begin{align*}
\limsup_{N\to\infty}\langle\Delta_N(\widetilde {\mathcal R}_\infty(\theta_1)-c_1(\chi_\infty\circ f^{n_{i_0}}(\chi R))),\phi\rangle&\leq(1-c_1c'')\langle S,\phi\rangle,\\
\limsup_{N\to\infty}\langle\Delta_N(\chi_\infty\circ f^{n_{i_0}}(\chi R)),\phi\rangle&\leq c''\langle S,\phi\rangle,
\end{align*}
for all $\phi\in\Pc(V).$ Hence, we have $\lim_{N\to\infty}\Delta_N(\chi_\infty\circ f^{n_{i_0}}(\chi R))=c''S\geq c'S.$ As $\chi R\geq\chi_\infty\circ f^{n_{i_0}}(\chi R)$ and $c'<c$ was arbitrary, any limit value $R'$ of $(\Delta_N\chi R)_{N\geq1}$ has to satisfy $R'\geq cS.$ It follows by Lemma \ref{le existence courant} that $R'=cS$ since $R'$ is positive of mass $c.$
\end{proof}
\begin{remark}\label{rk-domi}
If $R_\infty$ is a limit value of $(\Delta_n\omega^p)_{n\geq1}$ then Lemma \ref{le conv tronca} implies that for each attracting current $S$ in $\Cc_p(\Pb^k)$ there exists $c>0$ such that $R_\infty\geq cS.$
\end{remark}
\begin{lemma}\label{le preli a finitude des tau}
Let $S_1$ and $S_2$ be two different attracting currents of bidegree $(p,p),$ attracting on $V_1$ and $V_2$ respectively. If $S_1$ and $S_2$ are different then $V_1\cap\Jc_s\cap V_2=\varnothing$ and $\Cc_p(V_1)\cap\Cc_p(V_2)=\varnothing.$
\end{lemma}
\begin{proof}
We proceed by contraposition. Let assume that $V_1\cap\Jc_s\cap V_2\neq\varnothing.$ Let $\chi$ be a positive smooth function with compact support in $V_1\cap V_2$ and such that $c:=\langle T^s,\chi\omega^p\rangle>0.$ By Lemma \ref{le conv tronca} we must have $\lim_{N\to\infty}\Delta_N(\chi\omega^p)=cS_i$ for $i=1,2.$ It follows that $S_1=S_2.$ In the same way, if there exists $R\in\Cc_p(V_1)\cap\Cc_p(V_2)$ then we can use a regularization to obtain a smooth current $R'\in\Cc_p(V_1)\cap\Cc_p(V_2).$ Therefore, the definition of attracting current implies
$$S_1=\lim_{N\to\infty}\Delta_NR'=S_2.$$
\end{proof}

\begin{lemma}\label{le finitude cas ca}
Let $\eta>0.$ There exist at most finitely many different attracting currents of bidegree $(p,p)$ which are attracting on a trapping region $V$ such that $f(V_\eta)\subset V.$ Moreover, this number is bounded by a constant depending only on $\eta>0.$
\end{lemma}
\begin{proof}
Let $(S_i)_{i\in I}$ be a family of attracting $(p,p)$-currents such that $S_i$ is attractive on a trappings region $V_i$ with $f(V_{i,\eta})\subset V_i.$ Assume that they are pairwise distinct. By Lemma \ref{le preli a finitude des tau} $V_{i,\eta}\cap\Jc_s\cap V_{j,\eta}=\varnothing$ if $i\neq j.$ Finally, by Corollary \ref{coro p-pseudo} $V_i\cap\Jc_s\neq\varnothing.$ Hence, if $x_i\in V_i\cap\Jc_s$ then the balls $B(x_i,\eta/3)$ with $i\in I$ are pairwise disjoint which implies that $I$ is finite with $\card(I)\leq Vol(B(x,\eta/3))^{-1}.$
\end{proof}
The following technical but important result says that an attracting current in $\Cc_p(U)$ has to be attractive on a large trapping region.

\begin{lemma}\label{le courant avec prop de cv}
Let $U$ be a trapping region of codimension $p.$ If $S$ is an attracting current in $\Cc_p(U)$ then $S$ is attractive on $\Nc_S(r_U).$
\end{lemma}
\begin{proof}
Let $S\in\Ic_p(U)$ be an attracting current. By definition $S$ is attractive on a trapping region $V$ and in particular $S$ belongs to $\Ic_p(V).$ Possibly by exchanging $V$ by $U\cap V,$ we can assume that $V\subset U.$ Define $r_0$ to be the supremum of the number $r>0$ such that $S$ is attractive on $\Nc_S(r).$ Since $f(V)\Subset V,$ we have $\Nc_S(r_V)\subset V$ and therefore $r_0$ has to be positive with $r_0\geq r_V>0.$ Our goal is to prove that $r_0\geq r_U.$

Assume by contradiction that $r_0<r_U.$ First observe that $S$ is attractive on $\Nc_S(r_0)$ since $\Nc_S(r_0)=\cup_{r<r_0}\Nc_S(r).$ On the other hand, for each $r_0<r\leq r_U$ and each $R\in\Ic_p(U),$ Theorem \ref{th conv1} applied to $\Nc_R(r)$ gives the existence of an attracting current $\tau\in\Ic_p(\Nc_R(r))$ which is attractive on $\Nc_\tau(r).$ Since $f(\Nc_\tau(r)_{\eta(r)})\subset\Nc_\tau(r)$ with $\eta(r)\geq\eta(r_0)>0,$ Lemma \ref{le finitude cas ca} implies that the set of such currents is finite. We denote them by $\tau_1,\ldots,\tau_m.$ To summarize, for each $r_0<r\leq r_U$ and each $R\in\Ic_p(U)$ there exists $i\in\{1,\ldots,m\}$ such that $\tau_i\in\Ic_p(\Nc_R(r)).$ Moreover, as they are finitely many there is $r_1$ with $r_0<r_1<r_U$ such that $\tau_i$ is attractive on $\Nc_{\tau_i}(r_1)$ for each $1\leq i\leq m.$ Due to the definition of $r_0,$ it implies that $S$ is not equal to $\tau_i,$ $1\leq i\leq m.$

Let $\widetilde\phi$ be a smooth $(s,s)$-form on $\Pb^k$ such that $\ddc\widetilde\phi>0$ on $\overline U.$
Let $r>0$ be such that $r_0<r<r_1.$ Since $\overline{\Ic_S(r)}$ is a compact convex set, there exists $R\in\overline{\Ic_S(r)}$ such that $\widetilde c(r):=\max_{R'\in\overline{\Ic_S(r)}}\langle R',\widetilde\phi\rangle=\langle R,\widetilde\phi\rangle.$

The first step is to show that we can choose $R$ to be a convex combination of $\tau_1,\ldots,\tau_m.$ To this purpose, first notice that $R\in\overline{\Ic_S(r)}$ implies, by Lemma \ref{le monotonic}, that $\overline{\Ic_R(r)}\subset\overline{\Ic_S(r)}.$ Therefore, $\max_{R'\in\overline{\Ic_R(r)}}\langle R',\widetilde\phi\rangle=\widetilde c(r).$ Moreover, as we have observe above, the set $\Nc_R(r)$ has to support one of the current $\tau_i,$ for some $1\leq i\leq m.$ Assume for simplicity that it is $\tau_1.$ Let $\chi_1$ be a positive smooth function with compact support in $\Nc_{\tau_1}(r),$ bounded by $1$ and equal to $1$ on $f(\Nc_{\tau_1}(r)).$ In particular, $\chi_1\circ f\geq\chi_1.$ Since $\tau_1\in\Ic_p(\Nc_R(r))$ we have that $\Nc_{\tau_1}(r)\subset \Nc_R(r)$ and therefore 
$$c_1(r):=\max_{R'\in\overline{\Cc_R(r)}}\langle R',\chi_1T^s\rangle$$
has to be strictly positive. Again, by the continuity of $R'\mapsto R'\wedge T^s$ and the compactness of $\overline{\Cc_R(r)},$ this value $c_1(r)$ is reached by a current $R'$ in $\overline{\Cc_R(r)}.$ By Lemma \ref{le disque dans Cc_S}, there exists a sequence of structural disks $(\mathcal R_n)_{n\geq1}$ such that $\mathcal R_n(0)=R,$ for each $\theta\in\Db^*$ $\mathcal R_n(\theta)$ is the image by $\Lambda$ of a smooth current, belongs to $\overline{\Cc_R(|\theta|)}$ and $R'=\lim_{n\to\infty}\mathcal R_n(r).$ Now, let $n\geq1$ and let define
$$u_N(\theta):=\langle\Delta_N\mathcal R_n(\theta),\widetilde\phi\rangle.$$
They are subharmonic functions on $\Db$ such that 
$$u_N(0)=\langle R,\widetilde\phi\rangle=\widetilde c(r)\ \ \text{and}\ \ \limsup_{N\to\infty}u_N(\theta)\leq\widetilde c(r)$$ for all $\theta\in\overline{\Db_r}.$
This last point comes from the fact that each limit values of $\Delta_N\mathcal R_n(\theta)$ is in $\overline{\Ic_R(r)}$ if $|\theta|\leq r.$ Therefore, $\limsup_{N\to\infty}u_N(\theta)=\widetilde c(r)$ for almost all $\theta\in\Db_r.$ In particular, there exists $\theta_n\in\Db_r$ such that $\limsup_{N\to\infty}u_N(\theta_n)=\widetilde c(r)$ and 
$$\langle\mathcal R_n(\theta_n),\chi_1T^s\rangle\geq\langle\mathcal R_n(r),\chi_1T^s\rangle-1/n.$$
On the other hand, as $\chi_1\circ f\geq\chi_1$ we have that
$$\langle\Delta_N\mathcal R_n(\theta_n),\chi_1T^s\rangle\geq\langle\mathcal R_n(\theta_n),\chi_1T^s\rangle.$$
By taking a suitable subsequence of $(\Delta_N\mathcal R_n(\theta_n))_{N\geq1}$ we obtain at the limit a current $R_n$ such that $R_n\in\overline{\Ic_R(r)},$ $\langle R_n,\widetilde\phi\rangle=\widetilde c(r)$ and $\langle R_n,\chi_1T^s\rangle\geq\langle\mathcal R_n(r),\chi_1T^s\rangle-1/n.$ Finally, if $R_\infty$ is a limit value of $(R_n)_{n\geq1}$ then
\begin{itemize}
\item $R_\infty\in\overline{\Ic_R(r)},$
\item $\langle R_\infty,\widetilde\phi\rangle=\widetilde c(r)$ and
\item $\langle R_\infty,\chi_1T^s\rangle=c_1(r).$
\end{itemize}
The latter equality implies that $R_\infty\geq c_1(r)\tau_1.$ To be more precise, since $\mathcal R_n(\theta_n)$ is the image by $\Lambda$ of a smooth current and $\mathcal R_n(\theta_n)\geq\chi_1\mathcal R_n(\theta_n)$ we deduce from Lemma \ref{le conv tronca} that $R_n\geq\langle\mathcal R_n(\theta_n),\chi_1T^s\rangle\tau_1$ which implies at the limit that $R_\infty\geq c_1(r)\tau_1,$ as $\tau_1$ is attractive on $\Nc_{\tau_1}(r).$ If $c_1(r)=1$ then we have prove the fact that the value $\widetilde c(r)$ can be reach by a convex combination of $\tau_1,\ldots,\tau_m$ since $\tau_1=R_\infty$ and $\langle R_\infty,\widetilde \phi\rangle=\widetilde c(r).$ Otherwise, we can write $R_\infty=c_1(r)\tau_1+(1-c_1(r))R'_\infty$ where $R'_\infty\in\Ic_p(\Nc_R(r)).$ Now we claim that $\tau_1$ cannot belong to $\Ic_p(\Nc_{R'_\infty}(r)).$ Indeed, if $\tau_1$ is in $\Ic_p(\Nc_{R'_\infty}(r))$ then there exists a structural disk $\mathcal R'_\infty$ obtained as a convex combination of structural disks of the form
$$\Lambda\sigma_1(\theta)_*\Lambda\sigma_2(\theta)_*\cdots\Lambda\sigma_l(\theta)_*R'_\infty,$$
where $\sigma_{i}\colon\Delta\to W$ are holomorphic functions with $\sigma_i(0)=\id,$ such that
$$\langle \mathcal R'_\infty(\theta_1),\chi_1T^s\rangle=c'>0,$$
for some $\theta_1\in\Db_r.$ But, the structural disk $\mathcal R_\infty$ obtained exactly as $\mathcal R'_\infty$ except that $R'_\infty$ is exchanged by $R_\infty$ gives a current $\mathcal R_\infty(\theta_1)\in\overline{\Cc_{R_\infty}(r)}\subset\overline{\Cc_R(r)}$ such that
\begin{align*}\langle\mathcal R_\infty(\theta_1),\chi_1T^s\rangle&=c_1(r)\langle\tau_1(\theta_1),\chi_1T^s\rangle+(1-c_1(r))\langle\mathcal R'_\infty(\theta_1),\chi_1T^s\rangle\\
&=c_1(r)+(1-c_1(r))c'>c_1(r),
\end{align*}
which contradicts the definition of $c_1(r).$ Therefore, $\tau_1\notin\Ic_p(\Nc_{R'_\infty}(r)).$ However, as above $\Nc_{R'_\infty}(r)$ has to support one of the currents $\tau_1,\ldots,\tau_m.$ Since it cannot be $\tau_1,$ we can assume for simplicity that it is $\tau_2.$ With the same construction than above applied to $R_\infty,$ we obtain a current $R_{2,\infty}\in\overline{\Ic_{R_\infty}(r)}$ such that
$$\langle R_{2,\infty},\widetilde\phi\rangle=\widetilde c(r)\ \ \text{and}\ \ R_{2,\infty}=c_1(r)\tau_1+c_2(r)\tau_2+(1-c_1(r)-c_2(r))R'',$$
where $c_2(r)$ is positive and maximal. Again, the maximality of $c_1(r)$ and $c_2(r)$ implies that the set $\Nc_{R''}(r)$ cannot contains $\tau_1$ nor $\tau_2.$ By induction, after at most $m$ steps we obtain a current $R_r\in\overline{\Ic_S(r)}$ such that $\langle R_r,\widetilde\phi\rangle=\widetilde c(r)$ and
$$R_r=\sum_{i=1}^mc_i(r)\tau_i\ \ \textrm{with}\ \ c_i(r)\geq0,\ \ \sum_{i=1}^mc_i(r)=1,$$
i.e. $R_r$ is in the convex hull of $\tau_1,\ldots,\tau_m.$

Let $(r_l)_{l\geq2}$ be a sequence decreasing toward $r_0,$ bounded by $r_1$ and such that each sequence $(c_i(r_l))_{l\geq2}$ converges. Define $c_i(r_0):=\lim_{l\to\infty}c_i(r_l)$, $R_{r_0}:=\sum_{i=0}^mc_i(r_0)\tau_i=\lim_{l\to\infty}R_{r_l}$ and $\widetilde c(r_0):=\langle R_{r_0},\widetilde\phi\rangle=\lim_{l\to\infty}\widetilde c(r_l).$ We deduce from $\widetilde c(r):=\max_{R\in\overline{\Ic_S(r)}}\langle R,\widetilde\phi\rangle$ with $S\in\overline{\Ic_S(r)}$ that $\widetilde c(r_0)\geq\langle S,\widetilde\phi\rangle.$

Now, we claim that $\langle S,\phi\rangle\geq\langle R_{r_0},\phi\rangle$ for all smooth forms $\phi$ such that $\ddc\phi\geq0$ on $U.$ This gives a contradiction and thus concludes the proof that $r_0\geq r_U.$ To be more precise, observe that it implies $\langle S,\widetilde\phi\rangle=\langle R_{r_0},\widetilde\phi\rangle$ since $\widetilde c(r_0)\geq\langle S,\widetilde\phi\rangle.$ Moreover, if $\psi$ is a smooth $(s,s)$-form then there exists $M>0$ such that $\ddc(M\widetilde\phi\pm\psi)\geq0.$ Therefore, $\langle S,M\widetilde\phi\pm\psi\rangle\geq\langle R_{r_0},M\widetilde\phi\pm\psi\rangle$ and thus $\langle S,\psi\rangle=\langle R_{r_0},\psi\rangle$ which implies that $S=R_{r_0}.$ In particular, $\supp(\tau_1)\subset\supp(S)$ as $R_{r_0}=\sum_{i=0}^mc_i(r_0)\tau_i$ with $c_1(r_0)>0.$ It is a contradiction since $\tau_1$ is an attracting current with $\tau_1\neq S$ and $S$ is the unique attracting current in $\Cc_p(V).$

It remains to prove the claim. Let $\phi$ be a smooth form such that $\ddc\phi\geq0$ on $U.$ Let $l\geq2.$ Since $R_{r_l}$ belongs to $\overline{\Ic_S(r_l)}$ there exists, by Lemma \ref{le disque dans Cc_S}, a sequence of structural disks $(\mathcal R^l_n)_{n\geq0}$ such that $\mathcal R^l_n(0)=S,$ $R_{r_l}=\lim_{n\to\infty}\mathcal R^l_n(r_l),$ $\mathcal R^l_n(\theta)$ belongs to $\overline{\Cc_S(|\theta|)}$ and is the image of a smooth current by $\Lambda$ if $\theta\neq 0.$ For $N\geq1,$ we can define the structural disk $\mathcal R^l_{n,N}$ by $\mathcal R^l_{n,N}(\theta):=\Delta_N(\mathcal R^l_n(\theta)).$ In particular, as $\overline{\Cc_S(r_1)}\subset\Cc_p(U),$ the restriction of these disks to $\Db_{r_1}$ defined structural disks in $\Cc_p(U).$ Let $u_{n,N}^l$ be the subharmonic function defined on $\Db_{r_1}$ by $u_{n,N}^l(\theta):=\langle\mathcal R_{n,N}^l(\theta),\phi\rangle.$ The sequence $(u_{n,N}^l)_{n\geq0}$ is uniformly bounded thus there exists a subsequence $(u_{n_i,N}^l)_{i\geq0}$ which converges to a subharmonic function $u_{\infty,N}^l$ such that for all $\theta\in\Db_{r_1}$
\begin{equation}\label{eq-1}
u_{\infty,N}^l(\theta)\geq\limsup_{i\to\infty}\langle\mathcal R^l_{n_i,N}(\theta),\phi\rangle.
\end{equation}
In particular, $u_{\infty,N}^l(0)\geq\langle S,\phi\rangle$ and $u_{\infty,N}^l(r_l)\geq\langle\Delta_NR_{r_l},\phi\rangle=\langle R_{r_l},\phi\rangle.$ Moreover, outside a polar subset of $\Db_{r_1}$ the inequality \eqref{eq-1} is an equality hence for almost all $\theta\in\Db_{r_1}$ $u_{\infty,N}^l(\theta)=\langle\Delta_N\widetilde R(\theta),\phi\rangle$ where $\widetilde R(\theta)$ is a limit value of $(\mathcal R_{n_i}^l(\theta))_{i\geq0}.$ In the same way, there exists a limit value $u_l$ of $(u_{\infty,N}^l)_{N\geq1}$ such that $u_l(0)\geq\langle S,\phi\rangle,$ $u_l(r_l)\geq\langle R_{r_l},\phi\rangle.$ Moreover, for almost all $\theta\in\Db_{r_1},$ $u_l(\theta)=\langle\widehat R(\theta),\phi\rangle$ for some current $\widehat R(\theta)\in\overline{\Ic_S(|\theta|)}.$ On the other hand, as $S$ is the unique attracting current in $\Nc_S(r_0),$ we deduce from Remark \ref{rk max} that
$$\langle S,\phi\rangle=\max_{R'\in\Ic_p(\Nc_S(r_0))}\langle R',\phi\rangle.$$
Thus, $u_l(\theta)\leq\langle S,\phi\rangle$ for all $\theta\in \Db_{r_0}.$ Since $u_l(0)\geq\langle S,\phi\rangle,$ the maximum principle implies that $u_l(\theta)=\langle S,\phi\rangle$ for all $\theta\in\Db_{r_0}.$ Hence, if $u_\infty$ denotes the limit of a subsequence of $(u_l)_{l\geq2}$ then $u_\infty(\theta)=\langle S,\phi\rangle$ for all $\theta\in \Db_{r_0}.$ It is an easy consequence of the mean value inequality and of the upper semicontinuity of $u_\infty$ that $u_\infty(\theta)=\langle S,\phi\rangle$ also for $\theta\in\overline{\Db_{r_0}},$ and in particular $u_\infty(r_0)=\langle S,\phi\rangle.$ Therefore, by Hartogs' lemma
$$\langle R_{r_0},\phi\rangle\leq\lim_{l\to\infty}u_l(r_l)\leq u_\infty(r_0)=\langle S,\phi\rangle,$$
which was the claimed inequality.
\end{proof}

The following theorem is a direct consequence of Lemma \ref{le finitude cas ca} and Lemma \ref{le courant avec prop de cv}. It implies the finiteness in Theorem \ref{th finitude}.

\begin{theorem}\label{th finitude des tau}
Let $U$ be a trapping region of codimension $p.$  The set of attracting currents in $\Cc_p(U)$ is finite. Moreover, the cardinality of this set is bounded by a constant depending only on $\eta_U.$
\end{theorem}
\begin{proof}
If $S$ is an attracting current in $\Cc_p(U).$ By Lemma \ref{le courant avec prop de cv}, $S$ is attracting on $\Nc_S(r_U).$ On the other hand, we have seen in Section \ref{subsec trap} that $f(\Nc_S(r_U)_{\eta_U})\subset \Nc_S(r_U).$ Therefore, by Lemma \ref{le finitude cas ca} there exist at most $m$ such currents with $m=Vol(B(x,\eta_U/3)^{-1}.$
\end{proof}

\begin{remark}
In the same way than in Remark \ref{rk max}, we could have easily deduced from Theorem \ref{th conv1} that there are finitely many attractive currents $\tau_1,\ldots,\tau_m$ in $\Cc_p(U)$ such that
$$\max_{S\in\Ic_p(U)}\langle S,\phi\rangle=\max_{1\leq i\leq m}\langle\tau_i,\phi\rangle,$$
for all form $\phi$ in $\cali P(U).$ This observation would be suffisant for several results in the sequel but not for those about quasi-attractors and holomorphic families.
\end{remark}

We have seen that the attracting currents are extremal in $\Ic_p(U).$ However, as the example of an attracting cycle of period $2$ shows, they might not be extremal in the set of positive closed currents invariant by $\Lambda^2.$ This explains why the equilibrium measures $\tau\wedge T^s$ are not mixing in general. But we have the following result which is an important step toward the study of the set $\Dc_p(U).$ 

\begin{proposition}\label{prop tau toujours extremal}
There exists an integer $n_0\geq1$ such that if we exchange $f$ by $f^{n_0}$ then for all $n\geq1$ the set of attracting currents for $f^n$ in $\Cc_p(U)$ is equal to the set of attracting currents for $f$ in $\Cc_p(U).$
\end{proposition}
\begin{proof}
Let $\Ac_p(f^n,U)$ denote the set of attracting currents for $f^n$ in $\Cc_p(U).$ Since $f^n(U_{\eta_U})\subset U$ for all $n\geq1,$ Theorem \ref{th finitude des tau} says that the cardinality of $\Ac_p(f^n,U)$ is bounded independently of $n.$ Let $n_0\geq1$ be such that $\Ac_p(f^{n_0},U)$ has the maximal number of elements. We will show that for all $n\geq1,$ $\Ac_p(f^{n_0},U)\subset\Ac_p(f^{nn_0},U)$ which implies the proposition by maximality.

By definition, if $\tau\in\Ac_p(f^{n_0},U)$ then $\Lambda^{n_0}\tau=\tau.$ Therefore, the construction of Section \ref{subsec trap} applied to $f^{n_0}$ instead of $f$ gives a trapping region $\Nc^{n_0}_\tau(r_U)$ for $f^{n_0}$ such that $\tau$ is attracting on $\Nc^{n_0}_\tau(r_U)$ with respect to $f^{n_0}.$ Moreover, by Lemma \ref{le preli a finitude des tau} $\Nc^{n_0}_\tau(r_U)\cap\Jc_s\cap\Nc^{n_0}_{\tau'}(r_U)=\varnothing$ if $\tau'$ is another element of $\Ac_p(f^{n_0},U).$

Assume by contradiction that for some $n\geq2,$ there exists $\tau\in\Ac_p(f^{n_0},U)$ such that $\tau$ is not attracting for $f^{nn_0}.$ By Theorem \ref{th conv1}, there exists an attracting current $\sigma$ for $f^{nn_0},$ attracting on a trapping region $D_\sigma$ for $f^{nn_0}$ such that $D_\sigma\subset\Nc^{n_0}_\tau(r_U).$ Therefore, if $R$ is a continuous form in $\Cc_p(D_\sigma)$ then
$$\lim_{N\to\infty}\frac{1}{N}\sum_{i=0}^{N-1}\Lambda^{inn_0}R=\sigma.$$
On the other hand, the fact that $\tau$ is attracting on $\Nc^{n_0}_\tau(r_U)$ implies
$$\lim_{N\to\infty}\frac{1}{N}\sum_{i=0}^{N-1}\Lambda^{in_0}R=\tau.$$
Hence,
$$\tau=\frac{1}{n}\sum_{i=0}^{n-1}(\Lambda^{n_0})^i\sigma.$$
Since $\sigma\neq\tau,$ there is $1\leq i_1\leq n-1$ such that $\sigma_1:=(\Lambda^{n_0})^{i_1}\sigma$ is different from $\sigma.$ We claim that $\sigma_1$ is also an attracting current for $f^{nn_0}.$ Indeed, let $S$ be a continuous form in $\Cc_p(f^{n_0i_1}(D_\sigma)).$ As $f^{nn_0}(D_\sigma)\Subset D_\sigma,$ the continuous form
$$\widetilde S:=\frac{1}{d^{pn_0i_1}}(f^{n_0i_1})^*S$$
is supported in the basin, with respect to $f^{nn_0},$ of $D_\sigma.$ Therefore, by Lemma \ref{le conv tronca}
$$\lim_{N\to\infty}\frac{1}{N}\sum_{i=0}^{N-1}\Lambda^{inn_0}\widetilde S=\sigma.$$
Using that $\Lambda^{n_0i_1}\widetilde S=S,$ we obtain
$$\lim_{N\to\infty}\frac{1}{N}\sum_{i=0}^{N-1}\Lambda^{inn_0}S=\Lambda^{n_0i_1}\left(\lim_{N\to\infty}\frac{1}{N}\sum_{i=0}^{N-1}\Lambda^{inn_0}\widetilde S\right)=\Lambda^{n_0i_1}\sigma=\sigma_1,$$
which proves the claim.

As $\Nc^{n_0}_\tau(r_U)\cap\Jc_s\cap\Nc^{n_0}_{\tau'}(r_U)=\varnothing$ if $\tau'$ is another current in $\Ac_p(f^{n_0},U),$ this shows that the set $\Ac_p(f^{nn_0},U)$ has strictly more elements than $\Ac_p(f^{n_0},U)$ which contradicts the definition of $n_0.$
\end{proof}

\subsection{Structure of $\Dc_p(U)$}\label{subsec Dp}
In this subsection, we will show that possibly by exchanging $f$ by an iterate we can remove the Ces\`{a}ro mean in the definition of attracting currents. In particular, the convergence in Theorem \ref{th finitude} will follow as a consequence. To this purpose, we study the set $\Dc_p(U).$ Unlike $\Ic_p(U),$ this set doesn't change when we exchange $f$ by $f^n.$ As a first step, like in the case of $\Ic_p(U),$ to a current in $S\in\Dc_p(U)$ we associate a sequence of open sets $(\Vc_{S_n})_{n\geq0}.$ Since $S$ is not invariant it is neither the case for these sets.

For the rest of this section, we fix an attracting current $\tau$ of bidegree $(p,p).$ By definition, $\tau$ is attractive on a codimension $p$ trapping region $V.$ We assume that $\tau$ is attractive on $V$ with respect to any iterates of $f.$ It is the minimal requirement in order to remove the Ces\`{a}ro mean and it will turn out to be sufficient. Notice that, by Proposition \ref{prop tau toujours extremal}, up to change $f$ by an iterate this assumption is always satisfied.

Let $S\in\Dc_p(V).$ As we have seen in Section \ref{subsec dim}, there exists a sequence $(S_n)_{n\geq0}$ in $\Dc_p(V)$ such that $S_0=S$ and $\Lambda S_{n+1}=S_n.$ This sequence might not be unique a priori and the construction below depends on it. However, for each current $S\in\Dc_p(V)$ we choose such a sequence.
The set $\Vc_{S_n}$ is defined as the union of the support of all currents of the form
$$\Lambda\sigma_{1*}\Lambda\sigma_{2*}\cdots\Lambda\sigma_{l*}S_{n+l},$$
where $\sigma_i$ are elements of $B_W(r_V).$ Exactly with the same arguments than for the sets $\Nc_K(r),$ we can show that these sets are open and satisfy $f(\Vc_{S_{n+1}})\Subset \Vc_{S_n}.$ We have the following relation between $V,$ $\Nc_\tau(r_V)$ and $(\Vc_{S_n})_{n\geq0}.$
\begin{lemma}\label{le union}
There exists an integer $m\geq1$ such that $\Nc_\tau(r_V)\subset\cup_{n=0}^m\Vc_{S_n}\subset V$ for all $S\in\Dc_p(V).$
\end{lemma}
\begin{proof}
The inclusion $\Vc_{S_n}\subset V$ for all $n\geq0$ is obvious. It simply comes from the facts that $\supp(S_n)\subset V$ and $f(\sigma(V))\subset V$ for all $\sigma\in B_W(r_V).$

In order to prove $\Nc_\tau(r_V)\subset\cup_{n=0}^m\Vc_{S_n},$ let $\chi_1,\ldots,\chi_M$ be positive smooth functions with compact support in $V$ such that
\begin{itemize}
\item $\sum_{i=1}^M\chi_i\geq\mathbf 1_{A\cap\Jc_s},$ where $A$ is the attracting set associated to $V,$
\item for each $1\leq i\leq M,$ $c_i:=\langle T^s,\chi_i\omega^p\rangle>0,$
\item the diameter of $\supp(\chi_i)$ is less than $\eta_V/2.$
\end{itemize}
Since $\tau$ is attractive on $V,$ it follows from Lemma \ref{le conv tronca} that $\lim_{n\to\infty}\Delta_n\chi_i\omega^p=c_i\tau,$ with $c_i>0.$ But the support function $R\mapsto\supp(R)$ is lower semicontinuous with respect to the weak topology for $R$ and the Hausdorff metric on the compact sets. Therefore, there exists $m\geq1$ such that $\supp(\tau)\subset(\supp(\Delta_m\chi_i\omega^p))_{\eta_V/2},$ for all $1\leq i\leq M.$

Let $x$ be a point in $\Nc_\tau(r_V).$ Again, by definition there exist $y\in\supp(\tau)$ and $\sigma_1,\ldots,\sigma_l\in B_W(r_V)$ such that $x=f\circ\sigma_1\circ\cdots\circ f\circ\sigma_l(y).$ On the other hand, since $\sum_{i=1}^M\chi_i\geq\mathbf 1_{A\cap\Jc_s}$ and $\|S_{m+l+1}\wedge T^s\|=1,$ there is $1\leq i\leq M$ such that $\langle S_{m+l+1},\chi_i\omega^s\rangle>0.$ Moreover, we deduce from the invariance of $\supp(\tau)$ by $f$ and from $y\in\supp(\tau)\subset(\supp(\Delta_m\chi_i\omega^p))_{\eta_V/2}$ that there exist $z\in\supp(\chi_i),$ an integer $1\leq j\leq m$ and $\sigma\in B_W(r_V)$ such that $y=f(\sigma(f^j(z))).$ Finally, since the diameter of $\supp(\chi_i)$ is less than $\eta_V/2$ and $\langle S_{m+l+1},\chi_i\omega^s\rangle>0$ there is $w\in\supp(S_{m+l+1})$ and $\sigma'\in B_W(r_V)$ such that $\sigma'(w)=z.$ It follows that $x$ is in the support of
$$\Lambda\sigma_{1*}\cdots\Lambda\sigma_{l*}\Lambda\sigma_*\Lambda^j\sigma'_*S_{m+l+1}$$
which is contained in $\Vc_{S_{m-j}}$ with $0\leq m-j\leq m-1.$

Notice that we will not need in the sequel that $m$ is uniform in $S\in\Dc_p(V).$
\end{proof}
\begin{remark}
Since $\tau$ is attractive on $\Nc_\tau(r_V),$ we can exchange $V$ by $\Nc_\tau(r_V)$ and thus obtain $\Nc_\tau(r_V)=\cup_{n=0}^m\Vc_{S_n}.$
\end{remark}

If $\phi$ is in $\Pc(V)$ we define
$$c_\phi:=\max_{S\in\Dc_p(V)}\langle S,\phi\rangle,$$
which is a maximum because $\Dc_p(V)$ is compact. Observe that if $L^n\phi:=d^{-sn}f^{n*}\phi$ then $c_{L^n\phi}=c_\phi$ for all $n\geq1$ since $\langle S,L^n\phi\rangle=\langle \Lambda^nS,\phi\rangle$ and that $\Lambda^n\Dc_p(V)=\Dc_p(V).$ If $S$ is in $\Dc_p(V)$ then we denote by $\Pc_S$ the set of forms $\phi$ in $\Pc(V)$ such that $\langle S,\phi\rangle=c_\phi.$ Finally, we define for a given $\phi\in\Pc(V)$ the following objects inductively. Let $R_0\in\Dc_p(V)$ be a current such that $\langle R_0,\phi\rangle=c_\phi$ and let $I_0:=\{0\},$ $i(0):=0.$ For $n\geq1$ we define
$$i(n):=\inf\{l\geq i(n-1)+1\,|\,\exists R\in\Dc_p(V),\, \langle R,L^l\phi\rangle=\langle R,L^i\phi\rangle=c_\phi\ \forall i\in I(n-1)\}.$$
If $i(n)=+\infty,$ we set $I(n):=I(n-1)$ and $R_n:=R_{n-1}.$ Otherwise, we set $I(n):=I(n-1)\cup\{i(n)\}$ and we choose a current $R_n\in\Dc_p(V)$ such that $\langle R_n,L^i\phi\rangle=c_\phi$ for all $i\in I(n).$ Finally, $I(\infty):=\cup_{n\geq1}I(n)$ is the limit of the sets $I(n)$ and we choose a limit value $R_\infty\in\Dc_p(V)$ of the sequence $(R_n)_{n\geq1}.$ Observe that $L^i\phi$ belongs to $\Pc_{R_\infty}$ for all $i\in I(\infty)$ and that $I(\infty)$ is a maximal subset of $\Nb$ with such a current. The next step is to show that $I(\infty)$ contains an infinite arithmetic progression.
\begin{proposition}\label{prop arit}
Let $S$ be in $\Dc_p(V)$ and let $(S_n)_{n\geq0}$ be a sequence in $\Dc_p(V)$ such that $S_0=S$ and $\Lambda S_{n+1}=S_n.$ There exist an integer $\gamma\geq1$ and a current $R\in\Dc_p(V)$ such that both $\Pc_S$ and $\Pc_{S_\gamma}$ are contained in $\Pc_R.$ In particular, applied to $S=R_\infty$ it implies that $I(\infty)+\gamma\subset I(\infty).$
\end{proposition}
\begin{proof}
Let $m\geq1$ be the constant obtained from Lemma \ref{le union}. Let $S_{n_i+2m+1}$ be a subsequence of $S_{n+2m+1}$ which converges. We call the limit $S_{\infty,2m+1}$ and as $\Lambda S_{n+1}=S_n$ we have that $S_{n_i+l}$ converges to $S_{\infty,l}:=\Lambda^{2m+1-l}S_{\infty,2m+1}$ for all $0\leq l\leq 2m+1.$  Since $\Dc_p(V)$ is compact, all these currents belong to it.

It follows from Lemma \ref{le union} that $\Nc_\tau(r_V)\subset\cup_{i=0}^{m}\Vc_{S_{\infty,i}}$ and $\Nc_\tau(r_V)\subset\cup_{i=0}^{m}\Vc_{S_{\infty,m+1+i}}.$ Therefore, there exist $0\leq\gamma_1,\gamma_2\leq m$ such that $\Vc_{S_\infty,\gamma_1}\cap\Jc_s\cap\Vc_{S_{\infty,m+1+\gamma_2}}\neq\varnothing.$ Moreover, $\Jc_s$ is invariant and $f(\Vc_{S_{\infty,i+1}})\subset\Vc_{S_{\infty,i}}$ thus $\Vc_{S_\infty,0}\cap\Jc_s\cap\Vc_{S_{\infty,\gamma}}\neq\varnothing$ with $\gamma:=m+1+\gamma_2-\gamma_1\geq1.$ From this it is easy to see, by combining the construction in  Proposition \ref{prop regu} and in Lemma \ref{le disque dans Cc_S}, that there are a positive smooth function $\chi$ with compact support in $V$ and two structural disks $\{S_{\infty,0}(\theta)\}_{\theta\in\Db},$ $\{S_{\infty,\gamma}(\theta)\}_{\theta\in\Db}$ of center $S_{\infty,0}$ and $S_{\infty,\gamma}$ respectively such that
$$\langle T^s,\chi\omega^p\rangle=\alpha>0,\ \ S_{\infty,0}(\theta_0)\geq\chi\omega^p\ \ \textrm{and}\ \ S_{\infty,\gamma}(\theta_0)\geq\chi\omega^p,$$
for some $\alpha>0,$ $\theta_0\in\Db^*.$ Moreover, as in Proposition \ref{prop regu}, the map $R\mapsto R(\theta_0)$ defined by the construction above is continuous with respect to weak topology  and the $\mathcal C^\infty$ topology respectively. Therefore, possibly by slightly modifying $\chi$ we obtain that the corresponding structural disks $\{S_{n_i}(\theta)\}_{\theta\in\Db},$ $\{S_{n_i+\gamma}(\theta)\}_{\theta\in\Db}$ of center $S_{n_i}$ and $S_{n_i+\gamma}$ respectively satisfy
$$S_{n_i}(\theta_0)\geq\chi\omega^p\ \ \textrm{and}\ \ S_{n_i+\gamma}(\theta_0)\geq\chi\omega^p,$$
for $n_i$ large enough.

Since, the disks $\{\Lambda^{n_i}S_{n_i}(\theta)\}_{\theta\in\Db}$ and $\{\Lambda^{n_i}S_{n_i+\gamma}(\theta)\}_{\theta\in\Db}$ are centered at $S$ and $S_\gamma$ respectively, the domination above implies, by Lemma \ref{le domi-cons2}, that if $\alpha R$ is a limit value of $\Lambda^{n_i}\chi\omega^p$ then $R$ is in $\Dc_p(V)$ and $\langle R,\phi\rangle=c_\phi,$ $\langle R,\psi\rangle=c_\psi$ as soon as $\phi\in\Pc_S$ and $\psi\in\Pc_{S_\gamma},$ i.e. $\Pc_S\cup\Pc_{S_\gamma}\subset\Pc_R.$

To prove the last point, observe that by taking $S=R_\infty$ we obtain a current $R\in\Dc_p(V)$ such that $L^i\phi$ and $L^{i+\gamma}\phi$ both belong to $\Pc_R$ for all $i\in I(\infty).$ But as $I(\infty)$ is maximal, it turns out that $I(\infty)+\gamma\subset I(\infty).$
\end{proof}
Here is the only point where the fact that $\tau$ is attractive for all the iterates of $f$ is involved.
\begin{lemma}\label{le arit ok}
Let $\phi$ be in $\Pc(V).$
If there exist an integer $\gamma\geq1$ and a current $S\in\Dc_p(V)$ such that $L^{n\gamma}\phi\in\Pc_S$ for all $n\geq0$ then $\phi$ belongs to $\Pc_\tau.$
\end{lemma}
\begin{proof}
If $\langle S,L^{n\gamma}\phi\rangle=c_\phi$ for all $n\geq0$ then for all $N\geq1$
$$\left\langle\frac{1}{N}\sum_{n=1}^{N}\Lambda^{n\gamma}S,\phi\right\rangle=c_\phi.$$
Any limit value $S_\infty$ of $N^{-1}\sum_{n=1}^N\Lambda^{n\gamma}S$ is invariant by $\Lambda^\gamma.$ As $\tau$ is the unique attracting current in $V$ with respect to $f^\gamma,$ it follows from Remark \ref{rk max} that
$$\langle\tau,\phi\rangle\geq\langle S_\infty,\phi\rangle=c_\phi.$$
Therefore, since $c_\phi:=\max_{R\in\Dc_p(V)}\langle R,\phi\rangle$ and $\tau\in\Dc_p(V)$ we conclude that $\langle\tau,\phi\rangle=c_\phi$ and thus $\phi\in\Pc_\tau.$
\end{proof}
We can now prove the following theorem which gives the convergence in Theorem \ref{th finitude}.
\begin{theorem}\label{th convergence dans V}
If $R$ is a continuous form in $\Cc_p(V)$ then $\lim_{n\to\infty}\Lambda^nR=\tau.$
\end{theorem}
\begin{proof}
Let $\phi$ be in $\Pc(V)$ and let $I(\infty)$ and $R_\infty$ as above. By Proposition \ref{prop arit}, there exists $\gamma\geq1$ such that $I(\infty)+\gamma\subset I(\infty).$ In particular, as $0$ belongs to $I(\infty)$ we deduce that $n\gamma\in I(\infty)$ for all $n\geq0,$ i.e. $L^{n\gamma}\phi$ is in $\Pc_{R_\infty}$ for all $n\geq0.$ It follows from Lemma \ref{le arit ok} that $\phi$ is in $\Pc_\tau.$ Therefore, since $\phi\in\Pc(V)$ was arbitrary, we have $\Pc(V)=\Pc_\tau.$ 

From this, the proof is very similar to the one of Lemma \ref{le conv tronca}. Let $R$ be a continuous form in $\Cc_p(V).$ We choose a subsequence $(\Lambda^{n_i}R)_{i\geq0}$ which converges toward a current $R'.$ It has to satisfy $R'\in\Dc_p(V)$ and therefore $\langle R',\phi\rangle\leq c_\phi=\langle\tau,\phi\rangle$ for all $\phi\in\Pc(V)$ as $\Pc(V)=\Pc_\tau.$ Let $\chi$ be a positive smooth function, bounded by $1,$ equal to $1$ on $f(\Nc_\tau(r_V))$ and with support in $\Nc_\tau(r_V).$ In particular, $\chi\circ f\geq\chi.$ Moreover, arguing as in Proposition \ref{prop regu} and Lemma \ref{le disque dans Cc_S} we can construct a structural disk $\{\tau(\theta)\}_{\theta\in\Db}$ of center $\tau$ such that $\tau(\theta)>0$ on $\supp(\chi)$ if $|\theta|<1$ is large enough. If $\phi\in\Pc(V)$ then the sequence of subharmonic functions defined by
$$u_i(\theta):=\langle\Lambda^{n_i}\tau(\theta),\phi\rangle$$
verifies $\limsup_{i\to\infty}u_i(\theta)\leq c_\phi$ and $u_i(0)=c_\phi$ for all $i\geq0.$ Therefore, Lemma \ref{le domi-cons1} implies that $\limsup_{i\to\infty}u_i(\theta)=c_\phi$ for almost all $\theta\in\Db.$ In particular, possibly by exchanging $(n_i)_{i\geq0}$ by a subsequence, there exists $\theta_1\in\Db$ such that $\tau(\theta_1)>0$ on $\supp(\chi)$ and $\lim_{i\to\infty}\langle\Lambda^{n_i}\tau(\theta_1),\phi\rangle=c_\phi.$ Let $l\geq0$ be large enough and define $S:=d^{-pl}(f^l)^*(\tau(\theta_1)).$ As in Lemma \ref{le conv tronca}, we can use a regularization in $\Cc_p(f^{-l}(V))$ of $S$ in order to obtain a structural disk $\mathcal S$ such that $\lim_{n\to\infty}\langle\mathcal S(\theta),\phi\rangle=c_\phi$ and $\mathcal S(\theta)>0$ on $\supp(\chi\circ f^l)$ for all $\theta\in\Db^*.$ In particular, as $R$ is continuous, if $\theta_2\in\Db^*$ there is $c>0$ such that $c\chi\circ f^lR\leq\mathcal S(\theta_2).$ It follows that
$$\langle R',\phi\rangle=\lim_{i\to\infty}\langle\Lambda^{n_i}R,\phi\rangle\geq\lim_{i\to\infty}\langle\Lambda^{n_i}(\chi\circ f^lR),\phi\rangle=\lim_{i\to\infty}\langle\Lambda^{n_i-l}(\chi\Lambda^lR),\phi\rangle=c_lc_\phi,$$
where $c_l:=\langle\chi\Lambda^lR,T^s\rangle.$

On the other hand, since $\tau$ is attractive on $V,$ we have $\lim_{N\to\infty}\Delta_NR=\tau.$ Moreover, $\chi\circ f\geq\chi$ implies that $c_l$ is an increasing sequence and its Ces\`{a}ro mean converges to $\langle\tau,\chi T^s\rangle=1.$ Thus, $\lim_{l\to\infty}c_l=1$ and therefore $\langle R',\phi\rangle\geq c_\phi.$ As the equality holds for all $\phi\in\Pc(V)$ we have that $R'=\tau.$ The result follows since $R'$ was an arbitrary limit value of $(\Lambda^nR)_{n\geq1}.$
\end{proof}

\subsection{Speed of convergence}\label{subsec-speed}
In this subsection, we study the speed of convergence in Theorem \ref{th finitude} and its consequences. First we show that the convergence is uniform for uniformly bounded currents. We were not able to establish an exponential speed in general. However, we prove it for pluriharmonic observables which will allow us to apply techniques developed in \cite{t-attrac-speed} and \cite{da-ta} in order to obtain an exponential speed in some special cases.

As in the previous subsection, we fix an attracting current $\tau$ of bidegree $(p,p)$ which is attractive on a codimension $p$ trapping region $V$ and we assume that $\tau$ is attractive on $V$ with respect to any iterates of $f.$

\begin{proposition}\label{prop conv unif}
Let $M>0.$ If $(R_n)_{n\geq1}$ is a sequence of continuous forms in $\Cc_p(V)$ such that $\|R_n\|_\infty\leq M$ then
$$\lim_{n\to\infty}\Lambda^nR_n=\tau.$$
\end{proposition}
\begin{proof}
Using a dense sequence in $\Cc_p(V)$ it is easy to construct $S\in\Cc_p(\overline V)$ such that $\supp(R)\subset\supp(S)$ for all $R\in\Cc_p(V).$ Observe that $\overline V\subset f^{-1}(V)$ and therefore we can consider a regularization $\widetilde S$ of $S$ in $f^{-1}(V)$ obtained by Proposition \ref{prop regu}. It is a smooth form which satisfies $\widetilde S>0$ on $\supp(S).$ In particular, $\widetilde S>0$ on $\supp(R)$ for all $R\in\Cc_p(V).$ Since $\tau$ is attractive on $V,$ it is also attractive on $f^{-1}(V)$ thus $\lim_{n\to\infty}\Lambda^n\widetilde S=\tau.$ If $(R_n)_{n\geq1}$ is a sequence of continuous forms in $\Cc_p(V)$ which are uniformly bounded, there is $0<c<1$ such that $\widetilde S\geq cR_n$ for all $n\geq1.$ On the other hand, any limit value $R'$ of $(\Lambda^nR_n)_{n\geq1}$ is in $\Dc_p(V)$ and therefore has to satisfy $\langle R',\phi\rangle\leq\langle\tau,\phi\rangle$ for all $\phi\in\Pc(V)$ as we have seen in the proof of Theorem \ref{th convergence dans V}. The same is true for the limit values of $(1-c)^{-1}\Lambda^n(\widetilde S-cR_n)$ hence we can apply Lemma \ref{le domi-cons2} which implies that $\lim_{n\to\infty}\langle\Lambda^nR_n,\phi\rangle=\langle\tau,\phi\rangle$ for all $\phi\in\Pc(V)$ i.e. $\lim_{n\to\infty}\Lambda^nR_n=\tau.$
\end{proof}

We denote by $\mathcal H$ the set of continuous real $(s,s)$-forms $\phi$ on $V$ such that $\ddc\phi=0$ and $|\langle R-\tau,\phi\rangle|\leq1$ for all $R\in\Cc_p(V).$
\begin{proposition}\label{prop conv harmonique}
There exist two constants $C>0$ and $0<\lambda<1$ such that for all $R\in\Cc_p(V),$ $\phi\in\mathcal H,$ and $n\geq1$ we have
$$|\langle\Lambda^nR-\tau,\phi\rangle|\leq C\lambda^n.$$
\end{proposition}
\begin{proof}
First we consider the regularization $\{\tau(\theta)\}_{\theta\in\Db}$ of $\tau$ in $V$ obtained by Proposition \ref{prop regu}. If $\phi$ is in $\mathcal H$ then $u(\theta):=\langle\tau(\theta)-\tau,\phi\rangle$ defines a harmonic function on $\Db$ such that $|u|\leq1$ and $u(0)=0.$ Therefore, if $\theta_0\in\Db$ then by Harnack's inequality there exists $\alpha<1$ such that $|\langle\tau(\theta_0)-\tau,\phi\rangle|\leq\alpha$ for all $\phi\in\mathcal H.$

If $R\in\Cc_p(V)$ then we denote by $\mathcal R$ the structural disk obtained by regularizing $R$ in $f^{-1}(V).$ Observe that by Proposition \ref{prop regu} there exist $\theta_1\in\Db^*$ and $M>0$ such $\|\mathcal R(\theta_1)\|_\infty\leq M$ for all $R\in\Cc_p(V).$ For each $n\geq2,$ we denote by $\mathcal R_n(\theta)$ the regularization of $\Lambda^n(\mathcal R(\theta))$ in $V$ evaluated at $\theta.$ It is easy to check that $\mathcal R_n$ corresponds to the diagonal of a structural variety defined on $\Db^2$ in $\Cc_p(V)$ and thus, it is a structural disk in $\Cc_p(V)$ which center is $\Lambda^nR.$

Since $\tau$ is attractive on $f^{-1}(V),$ by Proposition \ref{prop conv unif} $\Lambda^n\mathcal R(\theta_1)$ converges to $\tau$ uniformly on $R\in\Cc_p(V).$ Therefore, the continuity in Proposition \ref{prop regu} implies that there exist $\theta_0\in\Db,$ $0<c<1$ and $N\geq2$ such that
$$\mathcal R_N(\theta_1)\geq c\tau(\theta_0)$$
for all $R\in\Cc_p(V).$ If $\phi$ is in $\mathcal H$ then the function $v(\theta):=\langle \mathcal R_N(\theta)-\tau,\phi\rangle$ is harmonic with $|v|\leq1$ and
\begin{align*}
v(\theta_1)=\langle\mathcal R_N(\theta_1)-\tau,\phi\rangle&=\langle(\mathcal R_n(\theta_1)-c\tau(\theta_0))-(1-c)\tau,\phi\rangle+c\langle\tau(\theta_0)-\tau,\phi\rangle\\
&\leq 1-c+\alpha c,
\end{align*}
where $\alpha<1$ is defined above. Thus, $\langle\mathcal R_N(\theta_1)-\tau,\phi\rangle\leq 1-c+\alpha c<1$ for all $\phi\in\mathcal H$ and $R\in\Cc_p(V)$ and therefore, by Harnack's inequality there exists $\widetilde\lambda<1$ such that $\langle\Lambda^NR-\tau,\phi\rangle\leq\widetilde\lambda.$ As $-\phi\in\mathcal H$ if $\phi\in\mathcal H,$ we obtain that $|\langle\Lambda^NR-\tau,\phi\rangle|\leq\widetilde\lambda$ for all $R\in\Cc_p(V),$ i.e. $\widetilde\lambda^{-1}L^N\phi\in\mathcal H.$ It follows that $\widetilde\lambda^{-n}L^{nN}\phi\in\mathcal H$ for all $\phi\in\mathcal H.$ In particular
$$|\langle\Lambda^{nN}R-\tau,\phi\rangle|\leq\widetilde\lambda^n$$
for all $R\in\Cc_p(V),$ $\phi\in\mathcal H$ and $n\geq1$ which implies the desired result.
\end{proof}

In \cite{t-attrac-speed} and \cite{da-ta}, the geometric assumptions of Dinh (HD) were used in order to obtain the existence of the attracting current $\tau$ but also to prove the exponential speed for pluriharmonic observables. As we now know that these two points are always satisfied, repeating almost word by word the proofs of these articles we obtain an exponential speed of convergence in two special cases. The first case concerns attracting sets of small topological degree introduced by Daurat \cite{daurat}. The map $f$ is said to be of small topological degree on $V$ if $\limsup_{n\to\infty}|f^{-n}(x)\cap V|^{1/n}<d^s$ for all $x\in\Pb^k.$ In the second setting, we assume that $V$ is strictly $q$-convex. We refer to \cite{henkin} and \cite{de-book} for this notion. We will use the conventions of \cite{henkin} but observe that $q$-convex domains in \cite{henkin} correspond to strongly $(k-q)$-convex ones in \cite{de-book}.

\begin{theorem}\label{th vitesse cas particulier}
Let $\tau$ be a current which is attractive on the codimension $p$ trapping region $V$ with respect to each iterate of $f.$ Assume that one of the following conditions is satisfied.
\begin{itemize}
\item[(1)] $p=1$ and $f$ is of small topological degree on $V.$
\item[(2)] $V$ is strictly $(p-1)$-convex and there exist two open sets $V_1$ and $V_2$ such that $V\subset V_1\subset V_2,$ $V_1$ is a deformation retract of a dimension $s$ complex manifold $L\subset V_1$ and $\|\bigwedge^{s+1}Df(z)\|<1$ for all $z\in V_2.$

\end{itemize}
Then there exist constants $c>0$ and $0<\lambda<1$ such that
$$|\langle\Lambda^nR-\tau,\phi\rangle|\leq c\lambda^n\|\phi\|_{\mathcal C^2}$$
for all $R\in\Cc_p(V)$ and all $\mathcal C^2$ test form $\phi.$ In particular, $\tau$ is the unique invariant current in $\Cc_p(V).$
\end{theorem}
\begin{proof}[Sketch of proof]
We only give the main ideas and we refer to \cite{da-ta} or \cite{t-attrac-speed} for details.

For (1), we can consider the unique function $u$ such that $\ddc u=R-\tau$ and $\langle\mu,u\rangle=0.$ If $u_n:=d^{-sn}f_*u$ then $\ddc u_n=\Lambda^nR-\tau.$ On one hand, when $x\in V$ the assumption of small topological degree implies that most of the points in $f^{-n}(x)$ are outside $V.$ On the other hand, an easy consequence of Proposition \ref{prop conv harmonique} is that $|u_n(y)|\leq c\lambda^n$ uniformly for $y\in\Pb^k\setminus V.$ These two facts and a volume estimate for sublevel sets of p.s.h functions implies that $(u_n)_{n\geq1}$ converges exponentially fast to $0$ in $L^1$ which is equivalent to the desired result. Observe that in that setting we also obtain exactly as in \cite{da-ta} that $\tau$ have continuous local potentials.

In the second case, we need a resolution of the $\ddc$-equation with estimates for $\ddc$-exact $(s+1,s+1)$-forms on $V.$ To this aim, we follow the method of Dinh-Nguyen-Sibony in \cite{dns-hori}. In short, we start with a real $\mathcal C^2$ $(s,s)$-form $\phi$ on $V_1$ and we use the deformation retraction to solve $d\xi=\ddc\phi$ with estimates and we can assume that $\xi$ is real, i.e. $\xi=\Xi+\overline\Xi$ where $\Xi$ is a $(s,s+1)$-form. As $d\xi$ is a $(s+1,s+1)$-form, it follows that $\overline\partial\Xi=0$ and $d\xi=\partial\Xi+\overline\partial\overline\Xi.$ The deformation retraction also implies that $H^{s,s+1}(V_1,\Cb)=0.$ Therefore, $\Xi$ is $\overline\partial$-exact on $V_1$ and thus on $V.$ Since $V$ is strictly $(p-1)$-convex we can use the version with estimates of the Andreotti-Grauert theory by Henkin-Leiterer \cite{henkin} to solve the $\overline\partial\Psi=\Xi$ with estimates on $V.$ Hence, if $\psi=-i\pi(\Psi-\overline\Psi)$ then $\ddc\psi=\partial\overline\partial(\Psi-\overline\Psi)=\partial\Xi+\overline\partial\Xi=\ddc\psi$ on $V.$

In a second time, let $R$ be in $\Cc_p(V)$ and let $\phi$ be a $\mathcal C^2$ $(s,s)$-form on $\Pb^k.$ The assumption on $\|\bigwedge^{s+1}Df(z)\|$ implies that $\|d^{-sn}\ddc f^{n*}\phi\|_{\infty,V_2}$ is small. Therefore, the resolution above gives the existence of a $(s,s)$-form $\psi_n$ on $V$ such that $\|\psi_n\|_{\infty,V}$ is small and $\ddc\psi_n=d^{-sn}\ddc f^{n*}\phi.$ Hence, we have
$$\langle\Lambda^{n+l}R-\tau,\phi\rangle=\langle\Lambda^lR-\tau,d^{-ns}f^{n*}\phi-\psi_n\rangle+\langle\Lambda^l R-\tau,\psi_n\rangle.$$
By Proposition \ref{prop conv harmonique}, the first term is small if $l$ is large enough and the second one is small since $\psi_n$ is small.
\end{proof}
\begin{remark}
1) The $(p-1)$-convexity assumption always holds in a larger trapping region. Indeed, as we have already observe if $V$ is a trapping region of codimension $p$ then the same holds for its $(p-1)$-pseudoconvex hull $\widehat V.$ By definition, $\widehat V=\Pb^k\setminus\supp(S)$ for some $S\in\Cc_{s+1}(\Pb^k).$ Using a regularization $S'$ of $S,$ we can obtain an open set $\widehat V':=\Pb^k\setminus\supp(S')$ with smooth boundary. By \cite{fs-oka}, it is $(p-1)$-pseudoconvex. Hence, a result of Matsumoto \cite{matsumoto} implies that $\widehat V'$ is $(p-1)$-convex. Reducing slightly $\widehat V'$ we obtain a codimension $p$ trapping region which contains $V$ and which is strictly $(p-1)$-convex.\\
2) The notion of small topological degree is well-adapted to codimension $1$ attracting sets. In codimension $p,$ we can consider trapping regions $V$ such that for each current $S\in\Cc_{s+1}(\Pb^k)$ we have $\limsup_{n\to\infty}\|f^{n*}S\|_V^{1/n}<d^s.$ Very likely, Theorem \ref{th vitesse cas particulier} should hold in that setting and $\tau$ should have continuous super-potential (see \cite{ds-superpot} for this notion). With our approach it is necessary to solve the $\ddc$-equation on $V$ which seems possible to handle with considering the first point of this remark.
\end{remark}

The second condition seems difficult to check at first sight but if the associated attracting set $A$ is a complex manifold of dimension $s$ then the existence of arbitrarily small neighborhoods as $V_1$ is easy. Moreover, in that case $A$ is $(p-1)$-complete thus by \cite[Theorem IX.2.13]{de-book} $A$ admits a fundamental family of strictly $(p-1)$-convex neighborhoods. If furthermore $f$ is sufficiently critical on $A$, e.g. if $Df$ vanishes in all directions normal to $A$, then $\bigwedge^{s+1}Df$ vanishes on $A$ and the condition (2) is fulfilled by some well-chosen neighborhoods $V\subset V_1\subset V_2$ of $A.$ Small perturbations of such maps give a large family of examples which satisfy this condition since both conditions in Theorem \ref{th vitesse cas particulier} are stable under perturbations. We give two special families of examples. The first one is inspired by \cite[Theorem 4.1]{bd-elliptic}.
\begin{example}
Let $g\colon\Pb^s\to\Pb^k$ be a holomorphic map such that $A_0:=g(\Pb^s)$ is smooth and of dimension $s.$ Let $I\subset\Pb^k$ be a linear subspace of dimension $p-1$ such that $I\cap A_0=\varnothing.$ If $L\subset\Pb^k$ is a linear subspace of dimension $s$ then we can consider the projection $\pi:\Pb^k\setminus I\to L.$ By identifying $L$ with $\Pb^s$ we obtain a map $f_0:=g\circ\pi\colon\Pb^k\setminus I\to\Pb^k$ with $f_0(A_0)=A_0$ and which has rank $s.$ In particular, $\bigwedge^{s+1}Df_0=0$ everywhere. A generic small perturbation of $f_0$ gives a holomorphic map $f\colon\Pb^k\to\Pb^k$ which has an attracting set $A$ near $A_0$ which satisfies condition (2) in Theorem \ref{th vitesse cas particulier}. Moreover, if for the perturbation we only use homogeneous polynomials $h_i$ such that $h_i=0$ on $A_0$ then $A=A_0.$
\end{example}
The second family concerns perturbations of invariant critical hypersurfaces. They are the counterparts in higher dimension of super-attractive fixed points in dimension $1.$
\begin{example}
Let $f_0$ be in $\Hc_d(\Pb^k).$ Assume there exists a smooth hypersurface $H$ which is $f_0$-invariant and included in the critical set of $f_0.$ Then, by \cite[Lemma 7.9]{fs-cdhd0} and \cite{stawiska}, $H$ is an attracting set for $f_0.$ Since $H$ is critical, any small perturbation of $f_0$ has an attracting set close to $H$ which satisfies (2) in Theorem \ref{th vitesse cas particulier}.
\end{example}

We now show that the equidistribution toward $\tau$ gives an estimate on the mass of $(f^n)^*\omega^{s+1}$ which stay near $\tau.$ In particular, if the equidistribution has exponential speed then it will imply that the measure $\nu=\tau\wedge T^s$ is hyperbolic, cf. Theorem \ref{th nu hyp}.
Recall that $\widetilde V$ denotes the $s$-pseudoconcave core of $V$ defined in Definition \ref{def pseudoconcave core}. It is also a trapping region which contains the current $\tau.$
\begin{proposition}\label{prop-estim-degre}
If $f$ and $V$ are as above then $\|(f^n)^*\omega^{s+1}\|_{\widetilde V}=o(d^{sn}).$ Moreover, let assume that for each $M>0$ there exist constants $\lambda<1$ and $c>0$ such that
$$|\langle\Lambda^n(S-R),\phi\rangle|\leq c\lambda^n,$$
for all $n\geq0,$ $\phi\in\Pc(V)$ and continuous forms $S,R\in\Cc_p(V)$ such that $\|S\|_\infty\leq M,$ $\|R\|_\infty\leq M.$ Hence, we have
$$\limsup_{n\to\infty}\left(\|(f^n)^*\omega^{s+1}\|_{\widetilde V}\right)^{1/n}\leq\lambda d^s<d^s.$$
\end{proposition}
\begin{proof}
Since $\widetilde V$ is a trapping region, it admits a small neighborhood $\widetilde V'$ such that $f(\widetilde V')\Subset\widetilde V\Subset\widetilde V'.$ Let $\chi$ be a positive smooth function with support in $\widetilde V'$ and such that $\chi=1$ on $\widetilde V.$ By the definition of $\widetilde V,$ if $\widetilde V'$ is small enough there exists a smooth form $S\in\Cc_p(\widetilde V')$ such that $S>0$ on $\overline{\widetilde V}.$ Therefore, if $\phi$ is a current, smooth on $\widetilde V',$ such that $\ddc\phi=\omega^{s+1}-T^{s+1}$ then we have
\begin{align*}
d^{-sn}\|(f^n)^*\omega^{s+1}\|_{\widetilde V}&\leq d^{-sn}\langle (f^n)^*\omega^{s+1},\chi\omega^{p-1}\rangle=d^{-sn}\langle\omega^{s+1}-T^{s+1},(f^n)_*(\chi\omega^{p-1})\rangle\\
&=\langle\Lambda^n\ddc(\chi\omega^{p-1}),\phi\rangle=\langle b\Lambda^n(S-R),\phi\rangle,
\end{align*}
where $b>0$ is a constant such that $bS\geq\ddc(\chi\omega^{p-1})$ and $R:=S-b^{-1}\ddc(\chi\omega^{p-1}).$ The fact that $\tau$ is attractive on $V$ implies that $\lim_{n\to\infty}\langle \Lambda^n(S-R),\phi\rangle=0.$ This proves the first point. Moreover, if we have an exponential speed of convergence
$$|\langle \Lambda^n(S-R),\phi\rangle|\leq c\lambda^n,$$
then obviously we obtain the second point in the proposition.
\end{proof}

\section{The equilibrium measures}\label{sec-measure}
In this section, we deduce several properties on the equilibrium measure $\nu_\tau:=\tau\wedge T^s$ associated to $\tau$ from the equidistruction results towards $\tau.$ Most of the arguments in this section are now standard. The upper bound for the entropy of $\nu_\tau$ directly comes from \cite{d-attractor} and \cite{dethelin-selles} following Gromov \cite{gromov-entro}. The proof of the lower bound and of the mixing property of $\nu_\tau$ are based on \cite{bs-3}. Finally, the estimates on the Lyapunov exponents follow \cite{dethelin-lyap}. As the proofs are almost identical, we mainly focus our attention on the differences and the writing will be concise.

In this section, as in the end of Section \ref{sec courant}, we fix an attracting current $\tau$ of bidegree $(p,p)$ which is attractive on a codimension $p$ trapping region $V$ and we assume that $\tau$ is attractive on $V$ with respect to every iterates of $f.$

\begin{theorem}\label{th-nu-mixing}
The measure $\nu_\tau$ is mixing.
\end{theorem}
\begin{proof}
Let $\phi$ and $\psi$ be two smooth real-valued functions. We have to show that
$$\lim_{n\to\infty}\langle\nu_\tau,\phi(\psi\circ f^n)\rangle=\langle\nu_\tau,\phi\rangle\langle\nu_\tau,\psi\rangle.$$
Since $\nu_\tau=\tau\wedge T^s,$ we have
$$\langle\nu_\tau,\phi(\psi\circ f^n)\rangle=\langle(\phi\tau)\wedge T^s,\psi\circ f^n\rangle=\langle(\Lambda^n\phi\tau)\wedge T^s,\psi\rangle.$$
In order to prove that $\Lambda^n(\phi\tau)$ converges to $\langle\tau\wedge T^s,\phi\rangle\tau$ we observe that the currents $S^\pm:=2\tau\pm\|\phi\|_\infty^{-1}\phi\tau$ are both positive and $c^\pm:=\langle S^\pm,T^s\rangle=2\pm\|\phi\|_\infty^{-1}\langle\nu_\tau,\phi\rangle.$ Therefore, if $\Psi$ belongs to $\Pc(V)$ then $\limsup_{n\to\infty}\langle\Lambda^nS^\pm,\Psi\rangle\leq c^\pm\langle\tau,\Psi\rangle.$ It follows easily that $\langle\Lambda^n(\phi\tau),\Psi\rangle$ converges to $\langle\nu_\tau,\phi\rangle\langle\tau,\Psi\rangle$ for all $\Psi\in\Pc(V),$ i.e. $\Lambda^n(\phi\tau)$ converges to $\langle\nu_\tau,\phi\rangle\tau.$ Thus the theorem follows from Lemma \ref{le conti ts}.
\end{proof}
\begin{theorem}\label{th-nu-entropy}
The measure $\nu_\tau$ has maximal entropy $s\log d$ on $V.$
\end{theorem}
\begin{proof}
Since $V\cap\Jc_{s+1}=\varnothing,$ Theorem \ref{th entropy julia} implies that $h_{\nu_\tau}(f)\leq s\log d.$ The proof of the lower bound is identical to the one in \cite{bs-3} (see also \cite{dethelin-selles}, \cite{d-attractor}) and is based on the following equidistribution result. Let $H$ be a linear subspace of $\Pb^k$ of dimension $s$ such that there exists a positive smooth function $\chi$ supported on $V$ with $c:=\langle [H]\wedge T^s,\chi\rangle>0.$ Lemma \ref{le conv tronca} implies that for a generic choice of $H,$ there exists an increasing sequence $n_i$ with $\lim_{i\to\infty}\Lambda^{n_i}(\chi[H])=c\tau.$ Therefore, by Lemma \ref{le conti ts} we have
$$\lim_{i\to\infty}\frac{1}{n_i}\sum_{j=0}^{n_i-1}(\Lambda^{n_i}\chi [H])\wedge L^{n_i-j}\omega^s=c\tau\wedge T^s=c\nu_\tau.$$
From this, we can follow the proof of \cite{bs-3}, based on Yomdin's result \cite{yomdin}, which gives in this setting that $h_{\nu_\tau}\geq s\log d.$
\end{proof}

\begin{theorem}\label{th-nu-lyap}
The measure $\nu_\tau$ has at least $s$ positive Lyapunov exponents.
\end{theorem}
\begin{proof}
Let $\chi_1\geq\cdots\geq\chi_k$ be the Lyapunov exponents of $\nu_\tau.$ Since $h_{\nu_\tau}>0,$ the Margulis-Ruelle inequality implies that $\chi_1>0.$ On the other hand, if $\log(\Jac f)\notin L^1(\nu_\tau)$ then $\chi_k=-\infty.$ In particular, if $\log(\Jac f)\in L^1(\nu_\tau)$ then we can apply the result of \cite{dethelin-lyap} and otherwise $\chi_1\neq\chi_k$ and we can use \cite{dupont-entro}. In both cases, we obtain that if $c$ is such that
$$\chi_1\geq\cdots\geq\chi_c>0\geq\chi_{c+1}\geq\cdots\geq\chi_k,$$
then $h_{\nu_\tau}\leq c\log d.$ Since $h_{\nu_\tau}=s\log d,$ it gives $c\geq s.$
\end{proof}
We were not able to prove that $\nu_\tau$ is hyperbolic in general. However, all known examples satisfy additional conditions which imply that $\nu_\tau$ has $p=(k-s)$ negative exponents. In particular, $\nu_\tau$ is hyperbolic. Here is a condition which holds for all these examples and, we believe, might also hold in the general case. Define
$$d_{s+1,loc}:=\limsup_{n\to\infty}\|f^{n*}\omega^{s+1}\|_V^{1/n}.$$
We state the following theorem without proof. The ideas come from \cite{dethelin-lyap} and a proof in our setting can be found in \cite[Section 5]{da-ta} when $s=k-1.$ The general case is identical.
\begin{theorem}\label{th nu hyp}
If $d_{s+1,loc}<d^s$ then each ergodic measure $\nu$ supported in $V$ with $h_\nu=s\log d$ is hyperbolic with $p$ exponents smaller or equal to
$$\frac{1}{2}\log\left(\frac{d_{s+1,loc}}{d^s}\right)$$
and $s$ exponents larger or equal to $2^{-1}\log d.$
In particular, it is the case for $\nu_\tau.$
\end{theorem}
Proposition \ref{prop-estim-degre} implies that if the speed of convergence toward $\tau$ is exponential then $d_{s+1,loc}<d^s.$ In particular, it is the case under the assumptions of Theorem \ref{th vitesse cas particulier}. However, observe that the second condition in this theorem gives easily by itself that $\nu_\tau$ has $p$ negative exponents.
\section{Properties of attracting currents and bifurcations}\label{sec-ac-bif}
In this section, we initiate a study of the interactions between the structure and the dynamics of an attracting set $A$ and its attracting currents. First, we define the irreducible components of $A$ and we give a relation between them and the support of the attracting currents. Our results in this direction are very partial but they show that an attracting set is far to be an arbitrary subset of $\Pb^k.$ In a second subsection, we study holomorphic families with a common trapping region and we show that attracting currents can be used to define new currents in the parameter space which may encode some bifurcations.
\subsection{Irreducible components}
Let $U$ be a trapping region of dimension $s$ and let $A$ be the associated attracting set. We proved in Section \ref{sec courant} that the set of attracting currents in $\Cc_p(U)$ is finite and non-empty. Then, could we recover the attracting set $A$ from the attracting currents? Is $A$ the union of the supports of these attracting currents? If we only consider those in $\Cc_p(U),$ it is not the case. The self-map $[z:w:t]\mapsto[z^2:w^2:t^2]$ of $\Pb^2=\Cb^2\cup L_\infty$ has $A:=(0,0)\cup L_\infty$ as an attracting set. The set $\Cc_1(A)$ has only one element, $[L_\infty]$, and $A\neq\supp([L_\infty]).$ The reason is that $A$ is not of pure dimension $1.$ It also has a component of dimension $0.$ Therefore, we also have to consider the attracting currents of smaller bidimensions, i.e. of higher bidegrees. However, this can create redundancies. Indeed, in the above example there are four attracting currents, $[L_\infty]$ and the Dirac masses at the three attractive points, and two of these points lie on $L_\infty.$ Thus, we only consider special attracting currents to define the \textit{irreducible components} of $A.$
\begin{definition}\label{def-irreducible-compo}
An irreducible component of $A$ of dimension $k-l$ is an attracting set $A_S$ associated to a trapping region $\Nc_S(r_U)$ where $S\in\Cc_l(U)$ is an attracting current such that $\Nc_S(r_U)$ has dimension $k-l$ and is not included in the $(k-l+1)$-pseudoconvace core of $U.$
\end{definition}
Following Theorem \ref{th finitude des tau} we deduce that an attracting set only has finitely many irreducible components. It is easy to check that when $A$ is algebraic, possibly by exchanging $f$ by an iterate, the algebraic irreducible components coincide with the ones defined above. In particular, $A$ is the union of its irreducible components and $A_S=\supp(S)$ for each component. Are these two statements still true in general? Although we were not able to prove or disprove them, the following results give a link between $A_S$ and $\supp(S).$ The first one has nothing to do with attracting current. We only use that the set $\Nc_S(r_U)$ is filled by structural disks centered at $S.$
\begin{proposition}\label{prop-intersection-hori-avec-courant-attractif}
Let $S$ be an attracting current of bidegree $(l,l)$ in $U.$ Let $R$ be a positive $\ddc$-closed $(k-l,k-l)$-current defined in $\Nc_S(r_U).$ Then each connected component of $\supp(R)\cap\Nc_S(r_U)$ intersects $\supp(S).$
\end{proposition}
The current $R$ can be chosen as the restriction to $\Nc_S(r_U)$ of an element of $\Cc_{k-l}(\Pb^k).$ That is why we emphasize with the intersection with $\Nc_S(r_U).$
\begin{proof}
Let $X$ be a connected component of $\supp(R)\cap\Nc_S(r_U).$ Assume by contradiction that $X\cap\supp(S)=\varnothing.$ Therefore, there exists a positive smooth function defined on $\Nc_S(r_U)$ such that $\chi=1$ on a neighborhood of $X$ and $\supp(\chi)\cap\supp(S)=\varnothing.$ The current $\chi R$ is non-zero, positive, $\ddc$-closed and its support is disjoint from $\supp(S).$ On the other hand, using the definition of $\Nc_S(r_U),$ it is easy to construct a structural disk $\mathcal S$ in $\Cc_l(\Nc_S(r_U))$ with $\mathcal S(0)=S,$ $\mathcal S(\theta)$ is smooth if $\theta\in\Db^*$ and such that there exists $\theta_0\in\Db^*$ with $\langle\chi R,\mathcal S(\theta_0)\rangle>0.$ Since $\chi R$ is $\ddc$-closed, Theorem \ref{th sh} implies that the function $u$ defined by $u(\theta):=\langle\chi R,\mathcal S(\theta)\rangle$ is harmonic. But, $u(0)=0$ as $\supp(\chi R)\cap\supp(S)=\varnothing.$ Hence, the maximum principle implies that $u$ vanishes on $\Db$ which contradicts $u(\theta_0)=\langle\chi R,\mathcal S(\theta_0)\rangle>0.$
\end{proof}
\begin{remark}
If $S\wedge R$ is well-defined then each connected component of $\supp(R)\cap\Nc_S(r_U)$ intersects $\supp(S\wedge R).$ In particular, if $R$ is the restriction of $T^{k-l}$ to $\Nc_S(r_U)$ then each connected component of $\Jc_{k-l}\cap\Nc_S(r_U)$ intersects $\supp(\nu_S)$ where $\nu_S:=S\wedge T^{k-l}.$ It is likely that these components are related to the stable manifolds of the measure $\nu_S.$
\end{remark}
\begin{corollary}\label{cor-support-de-tau}
Let $S$ be an attracting current of bidegree $(l,l)$ in $U.$ If for each $x\in A_S$ there exists a positive $\ddc$-closed $(k-l,k-l)$-current $R$ defined on a neighborhood of $A_S$ such that $\{x\}$ is a connected component of $\supp(R)\cap A_S$ then $A_S=\supp(S).$ In particular, if $l=1$ and the Hausdorff dimension of $A_S$ is strictly less than $2k-1$ then $A_S=\supp(S).$
\end{corollary}
\begin{proof}
By definition we have $\supp(S)\subset A_S.$ To prove that $A_S\subset\supp(S),$ let $x$ be in $A_S$ and assume there exist a neighborhood $\Omega$ of $A_S$ and a positive $\ddc$-closed $(k-l,k-l)$-current $R$ defined on $\Omega$ such that $x$ is an isolated point in $\supp(R)\cap A_S.$ As $\Omega\cap\Nc_S(r_U)$ is a neighborhood of the attracting set $A_S$, we can define by induction, for each $n\geq1,$ a decreasing family of trapping regions $(\Omega_n)_{n\geq1}$ such that $\Omega_n$ is included in $\Omega\cap\Nc_S(r_U)\cap(A_S)_{1/n}.$ In particular, $S$ is attractive on $\Omega_n.$ Therefore, by Proposition \ref{prop-intersection-hori-avec-courant-attractif} $\supp(S)$ intersects the connected component $X_n$ of $\overline{\Nc_S(r_{\Omega_n})}\cap\supp(R)$ which contains $x.$ These sets form a decreasing family of connected compact subsets of $\supp(R)$ which all intersect $\supp(S).$ Hence, $X_\infty:=\cap_{n\geq1}X_n$ is also a compact connected subset of $\supp(R)$ which intersects $\supp(S).$ Moreover, $X_n\subset\overline{(A_S)_{1/n}}$ thus $X_\infty\subset A_S.$ It follows that $X_\infty=\{x\}$ and $x\in\supp(S).$

The last point comes from the fact that if the Hausdorff dimension of $A_S$ is less than $2k-1$ then the Hausdorff dimension of a generic slice of $A_S$ by a line $H$ is less then $1$ and thus is totally disconnected. Hence, the result above applies with $R=[H].$
\end{proof}
This criteria applies to the following family.
\begin{example}\label{ex linear}
Let $g$ be in $\Hc_d(\Pb^s).$ If $k>s$ we can consider the endomorphism $f_0$ of $\Pb^k$ defined by
$$[z_0:\cdots:z_s:z_{s+1}:\cdots:z_k]\mapsto[g(z_0,\ldots,z_s):z_{s+1}^d:\cdots:z_k^d].$$
It has degree $d$ and the dimension $s$ linear spaces $L:=\{[z_0:\cdots:z_k]\in\Pb^k\ |\ z_{s+1}=\cdots=z_k=0\}$ is an attracting set for $f_0.$ Moreover, $f_0$ also preserves the pencil $\mathcal P$ of dimension $p$ linear spaces which contain $I:=\{[z_0:\cdots:z_k]\in\Pb^k\ |\ z_0=z_1=\cdots=z_s=0\}$ and pass through a point of $L.$ If $\rho>0$ is small enough, the set
$$U:=\left\{[z_0:\cdots:z_k]\in\Pb^k\ |\ \max_{s+1\leq i\leq k}|z_i|\leq\rho\max_{0\leq i\leq s}|z_i|\right\}$$
is a trapping region for $L$ such that the intersection of $U$ with an element of $\mathcal P$ is isomorphic to a ball in $\Cb^p.$ Let denote by $\Fc_{d,s}$ the set constituted by all $(g,R_{s+1},\ldots,R_k)$ where $g\in\Hc_d(\Pb^s)$ and $R_{s+1},\ldots,R_k$ are homogeneous polynomials of degree $d.$ For each $(g,R_{s+1},\ldots,R_k)\in\Fc_{d,s}$ and for $\epsilon=(\epsilon_{s+1},\ldots,\epsilon_k)\in\Cb^p$ we can associated to $\lambda:=(g,R_{s+1},\ldots,R_k,\epsilon)$ a map $f_\lambda$ which sends $[z_0:\cdots:z_s:z_{s+1}:\cdots:z_k]$ to
$$[g(z_0,\ldots,z_s):z_{s+1}^d+\epsilon_{s+1}R_{s+1}(z_0,\ldots,z_k):\cdots:z_k^d+\epsilon_{k}R_{k}(z_0,\ldots,z_k)].$$
If $\epsilon$ is close enough to $0$ then $f_\lambda$ is holomorphic, preserves the pencil $\mathcal P$ and admits $U$ as a trapping region. We denote by $A_\lambda$ the associated attracting set. If $\epsilon$ is small enough, following \cite[Proposition 2.17]{fs-example} we can prove that the Hausdorff dimension of $A_\lambda\cap H$ is strictly smaller that $1$ for every element $H$ in $\mathcal P.$ In particular, these sets are totally disconnected.
 On the other hand, there exists a unique attracting current $\tau_\lambda$ for $f_\lambda$ in $U.$ Therefore, Corollary \ref{cor-support-de-tau} implies that $A_\lambda=\supp(\tau_\lambda).$ Observe that in this situation, the existence of $\tau_\lambda$ can be deduced from \cite{d-attractor}.
\end{example}
\subsection{Families of attracting sets}
We will now consider holomorphic families of endomorphisms admitting a common trapping region $U\subset\Pb^k.$ To be more precise, let $M$ be a complex manifold and let $F\colon M\times\Pb^k\to\Pb^k$ be a holomorphic map such that $f_\lambda:=F(\lambda,.)$ has degree $d$ for all $\lambda\in M.$ We assume that there exists an open set $U$ such that $f_\lambda(U)\Subset U$ for all $\lambda\in M$ and we define $A_\lambda:=\cap_{n\geq0}f_\lambda^n(U).$ We are interested in the variation in the dynamics and the structure of $A_\lambda$ with $\lambda.$

The constant defined in Definition \ref{def-taille} depends on $\lambda$ and we denote it by $r_U(\lambda).$ For simplicity, we only consider the local setting where $M$ is biholomorphic to a ball and there exists $r_0>0$ such that $r_U(\lambda)\geq r_0.$ Moreover, if $M$ is chosen small enough then the two following properties are satisfied (see Lemma \ref{le transi} and above for notations):
\begin{itemize}
\item for all $\lambda,\lambda'\in M,$ $x\in\Pb^k$ and $\sigma\in B_W(r_0/2)$ there exists $\sigma'\in B_W(r_0)$ such that $f_\lambda\circ\sigma(x)=f_{\lambda'}\circ\sigma'(x),$
\item for all $\lambda,\lambda'\in M,$ if $V$ is an open subset such that $f_\lambda(V_{\eta(r_0/2)})\Subset V$ then $f_{\lambda'}(V)\Subset V.$
\end{itemize}

We first show that attracting currents form locally structural varieties. See \cite{ds-allupoly} and \cite{pham} for a similar approach for the equilibrium measure $\mu.$

\begin{theorem}\label{th-parametre}
Let $\lambda_0\in M.$ If $\tau_1(\lambda_0),\ldots,\tau_N(\lambda_0)$ denote the the attracting currents in $\Cc_p(U)$ with respect to $f_{\lambda_0}$ then there exist structural varieties $\{\tau_i(\lambda)\}_{\lambda\in M},$ $1\leq i\leq N,$ such that for all $\lambda\in M$ the set of attracting currents in $\Cc_p(U)$ with respect to $f_\lambda$ is exactly $\{\tau_1(\lambda),\ldots,\tau_N(\lambda)\}.$ In particular, the number of attracting currents in $\Cc_p(U)$ is constant.
\end{theorem}
\begin{proof}
To each attracting current $\tau_i(\lambda_0)$ is associated a trapping region $\Nc^{\lambda_0}_{\tau_i(\lambda_0)}(r_0)$ on which $\tau_i(\lambda_0)$ is attractive by Lemma \ref{le courant avec prop de cv}. Here we add a symbol $\lambda_0$ to point out the dependence on the parameter. Let $\lambda\in M.$ The assumption on the size of $M$ ensures that $f_\lambda\circ\sigma(\Nc^{\lambda_0}_{\tau_i(\lambda_0)}(r_0))\Subset \Nc^{\lambda_0}_{\tau_i(\lambda_0)}(r_0)$ for all $\sigma\in B_W(r_0/2)$. Therefore, by Theorem \ref{th conv1} there is at least one attracting current $S$ in $\Cc_p(\Nc^{\lambda_0}_{\tau_i(\lambda_0)}(r_0))$ for $f_\lambda.$ We first show that $S$ is unique.

Since $f_\lambda\circ\sigma(\Nc^{\lambda_0}_{\tau_i(\lambda_0)}(r_0))\Subset \Nc^{\lambda_0}_{\tau_i(\lambda_0)}(r_0)$ for all $\sigma\in B_W(r_0/2),$ the set $\Nc^\lambda_S(r_0/2),$ defined with respect to $f_\lambda,$ is included in $\Nc^{\lambda_0}_{\tau_i(\lambda_0)}(r_0)$. Moreover, Lemma \ref{le courant avec prop de cv} implies that $S$ is attractive on $\Nc^\lambda_S(r_0/2).$ On the other hand, the assumption on $M$ implies $f_{\lambda_0}(\Nc^\lambda_S(r_0/2))\Subset\Nc^\lambda_S(r_0/2).$ Again, Theorem \ref{th conv1} states that $f_{\lambda_0}$ possesses an attracting current in $\Nc^\lambda_S(r_0/2)\subset\Nc^{\lambda_0}_{\tau_i(\lambda_0)}(r_0)$ which therefore has to be equal to $\tau_i(\lambda_0).$ In particular, $\tau_i(\lambda_0)\in\Cc_p(\Nc^\lambda_S(r_0/2)).$ Hence, $S$ has to be unique since by Lemma \ref{le preli a finitude des tau} $\Cc_p(\Nc^\lambda_S(r_0/2))\cap\Cc_p(\Nc^\lambda_{S'}(r_0/2))=\varnothing$ if $S'$ is another attracting current in $\Cc_p(\Nc^{\lambda_0}_{\tau_i(\lambda_0)}(r_0))$ for $f_\lambda.$ Thus, we can set $\tau_i(\lambda):=S.$

To show that these families of currents form structural varieties, first observe that since $\tau_i(\lambda_0)\in\Cc_p(\Nc^\lambda_{\tau_i(\lambda)}(r_0/2))$ we can take a regularization $R$ of $\tau_i(\lambda_0)$ in the tubular neighborhood $\supp(\tau_i(\lambda_0))_{\eta(r_0/2)}.$ Therefore, by Lemma \ref{le conv tronca} we have
$$\lim_{n\to\infty}\frac{1}{n}\sum_{i=1}^{n}\frac{1}{d^{si}}(f_\lambda^i)_*R=\tau_i(\lambda).$$
Define $\widetilde f\colon M\times\Pb^k\to M\times\Pb^k$ by $\widetilde f(\lambda,z)=(\lambda,f_\lambda(z)).$ If $\pi\colon M\times\Pb^k\to\Pb^k$ denotes the projection on the second coordinate then the currents $\mathcal R:=\pi^*R$ and $\mathcal R_n$ with
$$\mathcal R_n:=\frac{1}{n}\sum_{i=1}^{n}\frac{1}{d^{si}}(\widetilde f^i)_*\mathcal R$$
define structural varieties parametrized by $M$ and they are supported in $M\times U.$ Since $U$ is weakly $p$-pseudoconvex, the set of such currents is relatively compact. Let $(\mathcal R_{n_i})_{i\geq1}$ be a subsequence which converges to some current $\mathcal R_\infty.$ If $\phi$ belongs to $\Pc(U)$ then the sequence of p.s.h function $u_{n_i}(\lambda):=\langle\mathcal R_{n_i}(\lambda),\phi\rangle$ converges to $u_\infty(\lambda):=\langle\mathcal R_\infty(\lambda),\phi\rangle.$ Therefore, $u_\infty(\lambda)\geq\limsup_{i\to\infty}u_{n_i}(\lambda)=\langle\tau_i(\lambda),\phi\rangle,$ with equality outside a pluripolar subset $E$ of $M.$ On the other hand, since $u_\infty$ is p.s.h, it satisfies $u_\infty(\lambda)=\limsup_{\lambda'\to\lambda,\lambda'\notin E}u_\infty(\lambda').$ Moreover, any limit value of $\tau_i(\lambda')$ when $\lambda'$ converges to $\lambda$ is $f_\lambda$-invariant and supported on $\Nc^\lambda_{\tau_i(\lambda)}(r_0).$ Therefore, by Remark \ref{rk max} we have
$$u_\infty(\lambda)=\limsup_{\lambda'\to\lambda,\lambda'\notin E}u_\infty(\lambda')\leq\langle\tau_i(\lambda),\phi\rangle\leq u_\infty(\lambda).$$
Hence, $\langle\tau_i(\lambda),\phi\rangle=\langle\mathcal R_\infty(\lambda),\phi\rangle$ for all $\lambda\in M$ and all $\phi\in\Pc(U),$ i.e. $\mathcal R_\infty$ is a structural variety which slices at $\lambda$ is $\tau_i(\lambda).$

To conclude, observe that for $\lambda\in M$ we can exchange the role of $\lambda_0$ and $\lambda.$ It proves that the set of attracting currents in $\Cc_p(U)$ for $f_\lambda$ is exactly $\{\tau_1(\lambda),\ldots,\tau_N(\lambda)\}.$
\end{proof}
\begin{remark}
We do not know if these structural varieties are continuous. However, as we have observe in the proof, any limit value of $\tau_i(\lambda)$ when $\lambda$ converges to $\lambda_0$ is $f_{\lambda_0}$-invariant. Therefore, if $\tau_i(\lambda_0)$ is the unique $f_{\lambda_0}$-invariant current in $\Nc^{\lambda_0}_{\tau_i(\lambda_0)}(r_0)$ then the structural variety $\{\tau_i(\lambda)\}_{\lambda\in M}$ is continuous at $\lambda_0.$ The same holds if there exists a current $R$ in $\Cc_p(\Nc^{\lambda_0}_{\tau_i(\lambda_0)}(r_0))$ and $\phi\in\Pc(U)$ with $\ddc\phi>0$ on $U$ such that
$$\left\langle\frac{1}{n}\sum_{i=0}^{n-1}\frac{1}{d^{si}}(f_\lambda^i)_*R,\phi\right\rangle\xrightarrow[n\to\infty]{}\langle\tau_i(\lambda),\phi\rangle$$
uniformly on $M.$
\end{remark}

One important consequence of Theorem \ref{th-parametre} is that the Lyapunov exponents of the equilibrium measures can be used to define p.s.h functions on $M.$ For simplicity, assume there exists a unique attracting current $\tau(\lambda)$ in $\Cc_p(U)$ with respect to $f_\lambda.$ Let $\nu(\lambda)$ be the equilibrium measure associated to $\tau(\lambda)$ and let denote by $\chi_1(\lambda)\geq\cdots\geq\chi_k(\lambda)$ the Lyapunov exponents of $\nu(\lambda).$ We also consider the numbers $\widetilde\chi_1(\lambda)\geq\cdots\geq\widetilde\chi_{k+1}(\lambda)$ such that $\{\widetilde\chi_1(\lambda),\ldots,\widetilde\chi_{k+1}(\lambda)\}=\{\chi_1(\lambda),\ldots,\chi_k(\lambda),\log d\}.$
\begin{corollary}\label{cor-sum-lyap}
For each $1\leq l\leq k+1$ the function $L_l(\lambda):=\sum_{i=1}^l\widetilde\chi_i(\lambda)$ is plurisubharmonic or identically equal to $-\infty$ on $M.$ In particular, the same holds for $\lambda\mapsto\sum_{i=1}^k\chi_i(\lambda).$
\end{corollary}
\begin{proof}
The lift $F_\lambda\colon\Cb^{k+1}\to\Cb^{k+1}$ of $f_\lambda$ extends holomophically to $\Pb^{k+1}.$ We still denote by $0$ the origin of $\Cb^{k+1}\subset\Pb^{k+1}.$ If $\pi\colon\Pb^{k+1}\setminus\{0\}\to\Pb^k$ is the canonical projection then $\pi\circ F_\lambda=f_\lambda\circ\pi.$ Since $U$ is a trapping region for $f_\lambda,$ it is easy to check that $\widetilde U:=\pi^{-1}(U)\cup B(0,r)$ in $\Pb^{k+1}$ is a trapping region for $F_\lambda$ if $r>0$ is small enough. It has codimension $p.$ Indeed, if $R$ belongs to $\Cc_p(U)$ then $\pi^*R$ extends to an element of $\Cc_p(\widetilde U).$ Hence $\Cc_p(\widetilde U)\neq\varnothing.$ On the other hand, the origin $0$ is super-attractive thus if $\widetilde T(\lambda)$ denotes the Green current of $F_\lambda$ and if $\widetilde R\in\Cc_q(\widetilde U)$ then $\widetilde R\wedge\widetilde T(\lambda)$ is supported in a compact subset of $\pi^{-1}(U).$ Therefore, $\pi_*(\widetilde R\wedge\widetilde T(\lambda))$ is a well-defined element of $\Cc_q(U)$ which implies that $q\geq p,$ i.e. $\widetilde U$ has codimension $p.$ Hence, by Theorem \ref{th-parametre} there exists a structural variety $\{\widetilde\tau(\lambda)\}_{\lambda\in M}$ in $\Cc_p(\widetilde U)$ such that $\widetilde\tau(\lambda)$ is an attracting current for $F_\lambda.$ Moreover, since the intersection of $\widetilde U$ with a small neighborhood of the hyperplane at infinity is a trapping region of dimension $s,$ it doesn't intersect $\supp(\widetilde T^{s+1}(\lambda)).$ Hence, the measure $\widetilde\nu(\lambda):=\widetilde\tau(\lambda)\wedge\widetilde T^{s+1}(\lambda)$ is supported in a fixed compact of $\Cb^{k+1}$ and for each $n\geq1$ and $1\leq l\leq k+1$ we can consider the functions
$$L_{n,l}(\lambda):=\langle\widetilde\nu(\lambda),n^{-1}\log\|\wedge^lD_xF^n_\lambda\|\rangle.$$
Since $\{\widetilde T(\lambda)\}_{\lambda\in M}$ is a structural variety with continuous potential (cf. \cite[Remark 1.33]{ds-lec}) the measures $\widetilde\nu(\lambda)$ also form a structural variety. Thus,  $L_{n,l}$ are p.s.h on $M$ or identically equal to $-\infty.$ Following the proof of \cite[Theorem 2.2]{pham}, we obtain that $L_l(\lambda):=\lim_{n\to\infty}L_{n,l}(\lambda)$ is equal to the sum of the $l$ largest Lyapunov exponents of $\widetilde\nu(\lambda)$ and is p.s.h (or equal to $-\infty).$ To prove that $L_l(\lambda)$ is equal to $\sum_{i=1}^l\widetilde\chi_i(\lambda),$ observe that $F_\lambda$ preserves the pencil of lines passing through $0.$ Its action on it is naturally identified to $f_\lambda$ and the action restricted to a line is of the form $t\mapsto t^d.$ Therefore, it is sufficient to prove that $\pi_*\widetilde\nu(\lambda)=\nu(\lambda).$

To this purpose, let $\widetilde\omega$ be a smooth element in $\Cc_1(\Pb^{k+1}\setminus B(0,r))$ and let $\widetilde\chi$ be a positive smooth function with compact support in $\Pb^{k+1}\setminus\{0\}$ which is equal to $1$ on $\Pb^{k+1}\setminus B(0,r).$ The equality $f_\lambda\circ\pi=\pi\circ F_\lambda$ gives for a current $R\in\Cc_p(U)$
\begin{align}\label{eq-pasbelle}
&\pi_*\left(\frac{1}{n}\sum_{i=1}^n\frac{F^i_{\lambda*}\pi^*R}{d^{i(s+1)}}\wedge \frac{F_\lambda^{m*}(\widetilde\omega\wedge\pi^*\omega^s)}{d^{m(s+1)}}\right)=\left(\pi_*\frac{1}{n}\sum_{i=1}^n\frac{F^i_{\lambda*}\pi^*R}{d^{i(s+1)}}\wedge\frac{F_\lambda^{m*}\widetilde\omega}{d^m}\right)\wedge\frac{f_\lambda^{m*}\omega^s}{d^{ms}}\notag\\
&=\frac{1}{n}\sum_{i=1}^n\pi_*\frac{F_{\lambda*}^{i}}{d^{is}}\left(\pi^*R\wedge \frac{F_\lambda^{(m+i)*}\widetilde\omega}{d^{(m+i)}}\right)\wedge\frac{f_\lambda^{m*}\omega^s}{d^{ms}}=\frac{1}{n}\sum_{i=1}^n\frac{f_{\lambda*}^{i}}{d^{is}}R\wedge\frac{f_\lambda^{m*}\omega^s}{d^{ms}}.
\end{align}
When $m=n$ go to infinity, by Lemma \ref{le conti ts}, the last term converges to $\nu_\lambda.$ On the other hand, since $\widetilde\omega$ has compact support in $\Pb^{k+1}\setminus B(0,r)$ and $0$ is super-attractive, the first term is equal to
$$\pi_*\left(\frac{1}{n}\sum_{i=1}^n\frac{F^i_{\lambda*}\widetilde\chi\pi^*R}{d^{i(s+1)}}\wedge \frac{F_\lambda^{m*}(\widetilde\omega\wedge\pi^*\omega^s)}{d^{m(s+1)}}\right)$$
where $\widetilde\chi\pi^*R$ is a smooth current in $\Cc_p(\widetilde U).$ If $R$ is chosen to be strictly positive on $\supp(\pi_*\widetilde\tau(\lambda)\wedge\widetilde T(\lambda))$ then each limit value $\widetilde R_\infty$ of $n^{-1}\sum_{i=1}^nd^{-i(s+1)}F_{\lambda*}^i\widetilde\chi\pi^*R$ satisfies $\widetilde R_\infty\geq c\widetilde\tau(\lambda)$ for some $c>0.$ Therefore, the equation \eqref{eq-pasbelle} implies that $\nu(\lambda)\geq c\pi_*(\widetilde\tau(\lambda)\wedge\widetilde T^{s+1}(\lambda)).$ Since the right hand side is invariant and $\nu(\lambda)$ is ergodic, we obtain the desired equality.
\end{proof}
We have several remarks about this construction. First, these p.s.h functions detect when the attracting set is critical. To be more precise, $\{L_{k+1}=-\infty\}$ corresponds to parameters where $\nu(\lambda)$ has at least one exponent equal to $-\infty$ while for parameters in $\{L_{s+2}=-\infty\}$ the measure $\nu(\lambda)$ has exactly $p$ exponents equal to $-\infty.$ We also have that on the open set $\{L_{s+2}<0\}$ the measure $\nu(\lambda)$ is hyperbolic. As observe by Pham \cite[Corollary 2.6]{pham}, for each $\delta\in\Rb$ the functions $L_{l,\delta}(\lambda):=\sum_{i=1}^l\max\{\widetilde\chi_i(\lambda),\delta\}$ are also p.s.h. Moreover, all the constructions in this subsection can be extended to larger parameter spaces $M\subset M'$ where the trapping region $U$ is not necessary preserved. In this case, the slices of the structural varieties defined in Theorem \ref{th-parametre} might not be attractive. However, the above functions are still p.s.h on $M'$ and if we know that the measure $\nu(\lambda)$ is always hyperbolic when $U$ is preserved then the function $L_{s+2,0}$ would satisfy
$$L_{s+2,0}(\lambda)=\log d+\sum_{i=1}^s\chi_i(\lambda)+\max(\chi_{s+1}(\lambda),0),$$
and would be a good way to identify the parameters for which the attracting set associated to $U$ ``bifurcates''. In particular, we can recover in this way the classical bifurcation locus when $k=1.$ Finally, let us notice that for an arbitrary endomorphism $f$ of $\Pb^k$ we can consider a limit value $R_\infty$ of $n^{-1}\sum_{i=1}^nd^{-is}f^i_*\omega^p.$ By Lemma \ref{rk-domi}, this current dominates all attracting currents of $f$ in $\Cc_p(\Pb^k)$ and it may be used to study them in family. In particular, it seems easy to prove that outside a plutipolar subset of $\Hc_d(\Pb^k)$ all the equilibrium measures considered in this paper have finite Lyapunov exponents.

It should exist several relations between the functions $L_l$ and the dynamics or the bifurcation of $f_\lambda$ which deserve to be study in the future.

We conclude this section with a remark about the dependency of an attracting set on the parameter.
\begin{proposition}
Assume there exists a unique attracting current $\tau(\lambda)$ in $\Cc_p(U)$ with respect to $f_\lambda.$ Let denote by $A_\lambda$ the attracting set associated to $U$ with respect to $f_\lambda.$ If $\lambda\mapsto\tau(\lambda)$ is continuous at $\lambda_0$ and if $A_{\lambda_0}=\supp(\tau(\lambda_0))$ then the map $\lambda\mapsto A_\lambda$ is continuous at $\lambda_0.$ 
\end{proposition}
\begin{proof}
The fact that attracting sets depend upper semicontinuously of the parameter is already true for continuous maps. On the other hand, the map $S\mapsto\supp(S)$ is lower semicontinuous for currents and $\supp(\tau(\lambda))\subset A_\lambda.$ Therefore, if $A_{\lambda_0}=\supp(\tau(\lambda_0))$ then $\lambda\mapsto A_\lambda$ has to be continuous at $\lambda_0.$
\end{proof}

In particular, this result applies to Example \ref{ex linear}. To be more precise, we have already seen that if $\epsilon\in\Cb^p$ is close enough to $0$ then the attracting set $A_\lambda$ defined in Example \ref{ex linear} has a unique attracting current $\tau(\lambda)$ such that $A_\lambda=\supp(\tau(\lambda)).$ Moreover, if $\epsilon$ is close enough to $0$ then the map $f_\lambda$ satisfies the second condition in Theorem \ref{th vitesse cas particulier}. Hence, the current $\tau(\lambda)$ is the unique invariant current in $U$ for $f_\lambda$ and thus $\lambda\mapsto\tau(\lambda)$ is continuous. Therefore, if $M$ is a small enough neighborhood of $\Fc_{d,s}\times\{0\}$ in $\Fc_{d,s}\times\Cb^p$ then $A_\lambda$ moves continuously on $M.$ Observe that since $A_\lambda=\supp(\tau(\lambda)),$ it implies that $\supp(\tau(\lambda))$ depend continuously of $\lambda.$ But it is not the case for $\supp(\nu_{\tau(\lambda)}).$ To see this, it is sufficient to choose an element $g_0\in\Hc_d(\Pb^s)$ such that the support of the equilibrium measure $\mu_g$ is not continuous at $g_0.$ For $\lambda:=(g,0,\ldots,0,0),$ $A_\lambda$ is equal to the linear space $L\simeq\Pb^s$ and $\nu_{\tau(\lambda)}$ can be identifies with $\mu_g.$ Therefore, the map $\lambda\mapsto\supp(\nu_{\tau(\lambda)})$ is not continuous at $\lambda_0:=(g_0,0,\ldots,0,0).$ Moreover, the same arguments where $g_0$ is a Lattès map show that it is not possible to follow continuously minimal quasi-attractors on $M.$

\section{Quasi-attractors and open questions}\label{seq-qa}
In this last section, we first consider the case of quasi-attractors. Indeed, using Lemma \ref{le courant avec prop de cv} we could transpose easily all our results about attracting currents to this case. Recall that, following \cite{hurley}, a quasi-attractor is an intersection of attracting sets. Such a set is said to be minimal if it is minimal for the inclusion in the set of all quasi-attractors. If $A$ is a quasi-attractor in $\Pb^k$ then we can define the dimension of $A$ exactly as in Definition \ref{def-dim}. We can now give the proof of the three corollaries stated in the introduction.

\begin{proof}[Proof of Corollary \ref{cor-qa-ca}]
When $A$ has dimension $s$ then, as we said in the introduction, there exists a decreasing sequence $(A_i)_{i\geq0}$ of attracting sets, all of dimension $s,$ such that $A=\cap_{i\geq0}A_i.$ Let $\Ac_p(A_i)$ be the set of attracting currents in $\Cc_p(A_i).$ By Theorem \ref{th finitude des tau}, these sets are finite and non-empty. The fact that $A_{i+1}\subset A_i$ implies $\Ac_p(A_{i+1})\subset\Ac_p(A_i)$ and therefore, there exists $i_0\geq0$ such that $\Ac_p(A_i)=\Ac_p(A_{i_0})$ if $i\geq i_0.$ In particular, $\supp(\tau)\subset\cap_{i\geq i_0}A_i$ for all $\tau\in\Ac_p(A_{i_0})$ i.e. all the attracting currents in $\Cc_p(A_{i_0})$ are supported in $A.$ The first point in the corollary follows.

To obtain the correspondence, assume there exist two different attracting currents $\tau_1$ and $\tau_2$ in $\Cc_p(A).$ If we denote by $A_{i_0,\tau_j}$ the irreducible component of $A_{i_0}$ associated to $\tau_j,$ $j\in\{1,2\},$ then $K_j:=A\cap A_{i_0,\tau_j}$ is a quasi-attractor of dimension $s$ such that $\Cc_p(K_1)\cap\Cc_p(K_2)=\varnothing.$ Therefore, if $A$ supports more than one attracting current, it is not minimal in the set of quasi-attractors of dimension $s.$ Hence, if we associate to each attracting current $\tau$ in $\Cc_p(A)$ the intersection $K_\tau$ of all quasi-attractors of dimension $s$ which contain $\supp(\tau),$ we obtain the desired correspondence.

Finally, it is a general result that there exist only countably many attracting sets. Each of them supports only finitely many attracting currents of maximal bidimension. Therefore, there are only countably many attracting currents. As a quasi-attractor which is minimal has to be minimal in some dimension, it supports a unique attracting current of maximal bidimension which characterizes it. Thus, the set of minimal quasi-attractors is at most countable.
\end{proof}

\begin{remark}\label{rk-qa-as}
Let $K$ be a minimal quasi-attractor of dimension $s$ and let $\tau$ be the associated attracting current. If there exists an attracting set $A$ which contains $\supp(\tau)$ and such that its irreducible component $A_\tau$ associated to $\tau$ satisfies $A_\tau=\supp(\tau)$ then $\supp(\tau)\subset K\subset A_\tau=\supp(\tau),$ i.e. $K$ has to be equal to the attracting set $A_\tau.$ Therefore, a study of the relation between an attracting set and the associated irreducible component can be a way to prove that all quasi-attractors are indeed attracting sets. In particular, by Corollary \ref{cor-support-de-tau} if $s=k-1$ and $K=\cap_{i\geq0}A_i$ is not an attracting set then the Hausdorff dimension of each $A_i$ is larger or equal to $2k-1.$
\end{remark}

\begin{proof}[Proof of Corollary \ref{cor-qa-nu}]
The proof that $\nu_\tau$ is ergodic is identical to the one of Theorem \ref{th-nu-mixing} but with Ces{\`a}ro means. The other points are direct consequences of Proposition \ref{prop tau toujours extremal}, Theorem \ref{th-nu-mixing}, Theorem \ref{th-nu-entropy} and Theorem \ref{th-nu-lyap}.
\end{proof}

\begin{proof}[Proof of Corollary \ref{cor-qa-final}]
Let $K$ be a minimal quasi-attractor of dimension $s$ for $f.$ Assume that $K$ is not minimal with respect to $f^n$ and let $K'\subset K$ be a minimal quasi-attractor for $f^n.$ Using the minimality of $K$ for $f,$ it is easy to see that $f^i(K')$ is also a minimal quasi-attractor for $f^n$ for all $1\leq i\leq n-1$ and $K=\cup_{i=0}^{n-1}f^i(K').$ Moreover, if $n$ is chosen minimal in order to have $f^n(K')=K'$ then the sets $f^i(K')$ are pairwise disjoint for $0\leq i\leq n-1.$ In particular, if a minimal quasi-attractor of dimension $s$ splits, all its components also has to be of dimension $s.$

On the other hand, if $U$ is a trapping region of dimension $s$ which contains $K$ then by Theorem \ref{th finitude des tau}, the number of attracting currents in $\Cc_p(U)$ for $f^n$ is uniformly bounded on $n.$ As a minimal quasi-attractor of dimension $s$ for $f^n$ supports such a current, their cardinality is also uniformly bounded on $n.$ Therefore, there exists a minimal $n_0\geq1$ such that $K$ splits into a maximal number of minimal quasi-attractors for $f^{n_0},$ $K=K_1\cup\cdots\cup K_{n_0}.$ In particular, each of them has to be of dimension $s$ and must be minimal for all iterates of $f^{n_0}.$ As we have seen, the sets $K_i$ are pairwise disjoint with $K_i=f^{i-1}(K_1).$ Each of then supports an attracting current $\tau_i$ with respect to $f^{n_0}$ and it is easy to check that $\tau_i=\Lambda^{i-1}\tau_1.$ Moreover, the fact that $K_i$ is minimal for all iterates of $f^{n_0}$ implies that $\tau_i$ is attractive for all iterates. From now on, for simplicity, let exchange $f$ by $f^{n_0}.$ Then, Theorem \ref{th convergence dans V} implies that for each $1\leq i\leq n_0$ there exists a trapping region $U_{K_i}$ of dimension $s$ such that for all continuous form $R\in\Cc_p(U_{K_i})$ we have $\tau_i=\lim_{n\to\infty}\Lambda^n R.$ In particular, $K_i$ intersects $U_{K_i}$ and thus $K_i\subset U_{K_i}$ by minimality. The properties of $\nu_{\tau_i}$ come from Corollary \ref{cor-qa-nu}.
\end{proof}

\subsection{Open questions}
To conclude, we ask the following questions. Some of them were already in \cite{d-attractor}. The first three ones are fundamental to understand the structure of quasi-attractors.
\begin{question}\label{q-qa-as}
Can there exist an endomorphism of $\Pb^k$ with a minimal quasi-attractor which is not an attracting set?
\end{question}
\begin{question}
Is an attracting set the union of its irreducible components defined in Definition \ref{def-irreducible-compo}?
\end{question}
As we observed in Remark \ref{rk-qa-as}, a positive answer to the next question will give a negative one to Question \ref{q-qa-as}.
\begin{question}
Let $A$ be an attracting set. Let $\tau$ be an attracting current supported on $A.$ Is the irreducible component of $A$ associated to $\tau$ equal to $\supp(\tau)$?
\end{question}

\begin{question}\label{q-hyp}
Do the equilibrium measures $\nu_\tau$ always hyperbolic?
\end{question}
If it is the case then the Closing Lemma of Katok (cf. \cite{dethelin-nguyen} for a version available in our setting) implies that each attracting set possesses a periodic orbit. A weak version of Question \ref{q-qa-as} should therefore be the following.
\begin{question}
Is there an endomorphism of $\Pb^k$ with a quasi-attractor which possesses no periodic orbit?
\end{question}
\begin{question}\label{q-unique}
Let $U$ be a trapping region of dimension $s$ with a unique attracting current $\tau.$ Is $\tau$ the unique invariant current in $\Cc_p(U)$? Is $\nu_\tau$ the unique measure of maximal entropy in $U$?
\end{question}
\begin{question}
Does the convergence towards an attracting current always happen with exponential speed?
\end{question}
This last question implies Question \ref{q-hyp} by Theorem \ref{th nu hyp} and Question \ref{q-unique} by \cite{daurat-2}.

\bibliographystyle{alpha}

\noindent
Universit\'e de Bourgogne Franche-Comt\'e, IMB, UMR CNRS 5584, 21078 Dijon Cedex, France.\\
  {\tt e-mail:johan.taflin$@$u-bourgogne.fr}
\end{document}